\documentclass[11pt, leqno]{amsart}
\usepackage{amsfonts,amssymb, amscd}
\usepackage{amsmath}
\usepackage{lmodern}
\usepackage{mathrsfs}
\usepackage{amsthm}
\usepackage{qsymbols}
\usepackage{mathtools}
\mathtoolsset{showonlyrefs}
\usepackage{dsfont}
\usepackage{graphicx}
\usepackage{latexsym}
\usepackage{chngcntr}
\usepackage[noadjust]{cite}
\usepackage{paralist}
\usepackage[parfill]{parskip}
\usepackage{bm}
\usepackage{tikz-cd}
\usepackage{nicefrac}
\usepackage{float}
\usepackage{csquotes}

\usepackage{xcolor}

\usepackage[shortlabels]{enumitem}
\setlist[enumerate,1]{label = \normalfont(\roman*), ref = (\roman*)}
\newtheorem{theorem}{Theorem}[section]
\newtheorem{lemma}[theorem]{Lemma}
\newtheorem{proposition}[theorem]{Proposition}
\newtheorem{corollary}[theorem]{Corollary}

\theoremstyle{definition}
\newtheorem{definition}[theorem]{Definition}
\newtheorem{remark}[theorem]{Remark}
\newtheorem{example}[theorem]{Example}
\newtheorem{agreement}{Agreement}
\usepackage[plainpages=false,pdfpagelabels,
backref=page,
citecolor=red]{hyperref}
\usepackage{xcolor}
\hypersetup{
 colorlinks,
 linkcolor={cyan!90!black},
 citecolor={magenta},
 urlcolor={green!40!black}
} %
\usepackage{doi}
\usepackage{underscore}
\newcommand{\R}{\mathbb{R}}
\newcommand{\C}{\mathbb{C}}
\renewcommand{\L}{\mathrm{L}}
\newcommand{\Lloc}{\mathrm{L}_{\mathrm{loc}}}
\newcommand{\Lx}{\mathrm{L}_{\vphantom{t}x}}
\newcommand{\Wx}{\mathrm{W}_{\vphantom{t}x}}
\newcommand{\Hx}{\mathrm{H}_{\vphantom{t}x}}
\newcommand{\Lt}{\mathrm{L}_t}
\newcommand{\Wt}{\mathrm{W}_t}
\newcommand{\Cont}{\mathrm{C}}
\newcommand{\Contx}{\mathrm{C}_{\vphantom{t}x}}

\newcommand{\Contt}{\mathrm{C}_t}

\newcommand{\A}{\mathrm{A}}
\newcommand{\e}{\mathrm{e}}

\renewcommand{\d}{\,\mathrm{d}}

\newcommand{\I}{\mathrm{I}}
\newcommand{\II}{\mathrm{II}}

\newcommand{\eps}{\varepsilon}
\newcommand{\B}{\mathrm{B}}
\newcommand{\Q}{\mathrm{Q}}
\renewcommand{\H}{\mathrm{H}}

\newcommand{\E}{\mathcal{E}} %

\newcommand{\Sec}{\mathrm{S}}
\newcommand{\cB}{\mathcal{B}}
\newcommand{\cC}{\mathcal{C}}

\newcommand{\cL}{\mathcal{L}}

\newcommand{\FT}{\mathcal{F}}

\newcommand{\Rbound}{\mathcal{R}}
\newcommand{\SP}{\, |\,}
\newcommand{\MO}{\mathcal{M}}

\newcommand{\z}{\mathbf{z}}
\renewcommand\Re{\operatorname{Re}}

\newcommand{\ind}{\mathbf{1}}

\renewcommand{\tt}{{t^*}}

\newcommand{\SetSep}{\; ; \;}

\DeclareMathOperator{\dist}{d}

\DeclareMathOperator{\dom}{D}

\let\div\olddiv
\DeclareMathOperator{\div}{div}
\def\Xint#1{\mathchoice
{\XXint\displaystyle\textstyle{#1}}%
{\XXint\textstyle\scriptstyle{#1}}%
{\XXint\scriptstyle\scriptscriptstyle{#1}}%
{\XXint\scriptscriptstyle%
\scriptscriptstyle{#1}}%
\!\int}
\def\XXint#1#2#3{{\setbox0=\hbox{$#1{#2#3}{%
\int}$ }
\vcenter{\hbox{$#2#3$ }}\kern-.6\wd0}}
\def\barint{\,\Xint-} %
\newcommand{\sysalpha}{k} %
\newcommand{\sysbeta}{\ell}
\newcommand{\shift}{\kappa} %
\newcommand{\M}{M} %

\title[Weighted non-autonomous $\L^q(\L^p)$ maximal regularity for complex systems]{Weighted non-autonomous $\boldsymbol{\L^q(\L^p)}$ maximal regularity\\ for complex systems under mixed regularity\\ in space and time}
\author{Sebastian Bechtel}
\address{Delft Institute of Applied Mathematics, Delft University of Technology, P.O. Box 5031, 2600 GA Delft, The Netherlands}
\email{s.bechtel@tudelft.nl}

\makeatletter
\@namedef{subjclassname@2020}{\textup{2020} Mathematics Subject Classification}
\makeatother

\dedicatory{}

\hypersetup{
pdftitle={Weighted non-autonomous L^q(L^p) maximal regularity for complex systems under mixed regularity in space and time},
pdfauthor={Sebastian Bechtel},
pdfcreator={LaTeX with hyperref (gitREF, gitSHA)},
pdfkeywords={non-autonomous maximal regularity, Lions problem, second-order elliptic systems, commutator estimates, weighted estimates, pseudo differential operators},
pdfsubject={MSC2020 Primary: 35B65. Secondary: 35J47, 35B45, 47D06.},
}

\begin{document}
\begin{abstract}
We show weighted non-autonomous $\L^q(\L^p)$ maximal regularity for families of complex second-order systems in divergence form under a mixed regularity condition in space and time. To be more precise, we let $p,q \in (1,\infty)$ and we consider coefficient functions in $\Contt^{\beta + \varepsilon}$ with values in $\Contx^{\alpha + \varepsilon}$ subject to the parabolic relation $2\beta + \alpha = 1$.  If $p < \nicefrac{d}{\alpha}$, we can likewise deal with spatial $\Hx^{\alpha + \eps, \nicefrac{d}{\alpha}}$ regularity. The starting point for this result is a weak $(p,q)$-solution theory with uniform constants. Further key ingredients are a commutator argument that allows us to establish higher a priori spatial regularity, operator-valued pseudo differential operators in weighted spaces, and a representation formula due to Acquistapace and Terreni. Furthermore, we show $p$-bounds for semigroups and square roots generated by complex elliptic systems under a minimal regularity assumption for the coefficients.
\end{abstract}
\maketitle

\section{Introduction}
\label{Sec: Introduction}

Fix a finite time $T>0$ and a dimension $d \geq 1$. Dong and Kim studied in a series of articles~\cite{Dong-Kim-ARMA, Dong-Kim-JFA, Dong-Kim-CoV, Dong-Kim-TAMS} solvability of the parabolic system in divergence form
\begin{align}
	\label{Eq: problem introduction DK}
	\partial_t u - \div_x A(t,x) \nabla_x u + \kappa u = f + \div_x F \quad \mathrm{in} \quad (-\infty, T) \times \R^d.
\end{align}
Here, $\kappa > 0$ is sufficiently large and $A \colon (-\infty,T) \times \R^d \to \C^{dm\times dm}$ satisfies $\| A \|_\infty \leq \Lambda$ and is elliptic in the following sense: there exists $\lambda > 0$ such that
\begin{align}
	\label{Eq: ellipticity}
	\sum_{\sysalpha,\sysbeta=1}^m \Re (A(t,x)^{\sysalpha \sysbeta} \xi^\sysalpha \SP \xi^\sysbeta) \geq \lambda |\xi|^2 \qquad (\xi \in \C^{dm}).
\end{align}
The number $m$ is the size of the system.
It turns out that a $\textrm{VMO}$ condition for $A$ (made precise in Lemma~\ref{lem: integral inequality}) is sufficient to guarantee unique solvability in the class $\Lt^q(\Wx^{1,p})$ for right-hand sides $f,F \in \Lt^q(\Lx^p)$ with $p,q \in (1,\infty)$. Given a parabolic Muckenhoupt weight, they also treat weighted estimates. We restrict our attention to temporal Muckenhoupt weights $w\in A_q$ in the sequel.

On the finite time interval $(0,T)$ their result implies well-posedness of the problem
\begin{align}
	\label{Eq: problem introduction weak}
	\begin{split}
	\partial_t u - \div_x A(t,x) \nabla_x u &= f + \div_x F \quad \mathrm{in} \quad (0, T) \times \R^d, \\
	u(0) &= 0.
	\end{split}
\end{align}
Observe that the right-hand side is in the class $\Lt^q(w;\Wx^{-1,p})$. It follows from the equation that $\partial_t u \in \Lt^q(w;\Wx^{-1,p})$ likewise. In other words,~\eqref{Eq: problem introduction weak} has maximal regularity over $\Lt^q(w;\Wx^{-1,p})$. If $f=0$, then the right-hand side is from the class $\Lt^q(w;\dot{\mathrm{W}}_x^{-1,p})$ and it follows again from the equation that $\partial_t u \in \Lt^q(w;\dot{\mathrm{W}}_x^{-1,p})$. However, if $F = 0$, then the right-hand side belongs to $\Lt^q(w;\Lx^p)$, but the higher regularity $\partial_t u \in \Lt^q(w;\Lx^p)$ of the time derivative is not known. It is the purpose of this article to investigate under which additional conditions one can show the improved regularity $\partial_t u \in \Lt^q(w;\Lx^p)$ for the problem
\begin{align}
	\label{Eq: non-autonomous equation}
	\tag{P}
	\begin{split}
		\partial_t u - \div_x A(t,x) \nabla_x u &= f, \qquad \mathrm{in} \quad (0, T) \times \R^d, \\
		u(0) &= 0.
	\end{split}
\end{align}
To make the notation more precise, we define for each fixed $\tt \in (0,T)$ an elliptic operator in divergence form in the following way:
consider the bounded sesquilinear form\footnote{Here, $\nabla_x$ denotes the gradient in the variable $x$. For the sake of readability, let us agree to omit the underlying sets $(0,T)$ and $\R^d$ in the notation of function space; instead, we will indicate the underlying set by the indices $t$ and $x$. For instance, we will simply write $\Wx^{1,2}$ instead of $\Wx^{1,2}(\R^d)$ and so on.}
\begin{align}
	a_\tt \colon \Wx^{1,2} \times \Wx^{1,2} \to \C, \qquad a_\tt(u,v) = \int_{\R^d} A(\tt, x) \nabla_x u(x) \cdot \overline{\nabla_x v(x)} \d x.
\end{align}
Using the form  $a_\tt$, we define the operator
\begin{align}
	\cL_\tt \colon \Wx^{1,2}\to \Wx^{-1,2} \quad \text{via} \quad \langle \cL_\tt u, v \rangle_{\Wx^{-1,2},\Wx^{1,2}} = a_\tt(u,v) \qquad (u,v\in \Wx^{1,2}).
\end{align}
Here, $\Wx^{-1,2}$ is the space of conjugate-linear functionals on $\Wx^{1,2}$. Eventually, the operators $-\div_x A(t,x) \nabla_x$ in~\eqref{Eq: non-autonomous equation} are defined by $\cL_t$.

In the case $p=q=2$ and with $w=1$ this question is known as \emph{Lions' maximal regularity problem} and was investigated by many authors~\cite{AO19,AE16,DZ17,Fackler15,HO15,OS10}, see also the survey article~\cite{Lions-Survey}.
For counterexamples that highlight the need of a certain regularity we refer to~\cite{Fackler16,BMV}.
Our main result reads as follows.
\begin{theorem}
	\label{Thm: main result}
	Let $p,q \in (1,\infty)$, $w\in \A_q$, $\alpha,\beta,\eps > 0$ such that $2\beta+\alpha=1$. For the coefficient function $A$, assume that
	\begin{align}
		A \in \begin{cases}
			\Cont_t^{\beta + \eps}(\Hx^{\alpha + \eps, \nicefrac{d}{\alpha}}), & \text{if } p < \nicefrac{d}{\alpha}, \\
			\Cont_t^{\beta + \eps}(\Contx^{\alpha + \eps}), & \text{else}.
		\end{cases}
	\end{align}
	Then, given $f\in \Lt^q(w;\Lx^p)$, the unique weak $(p,q)$-solution $u$ of~\eqref{Eq: non-autonomous equation} satisfies $\partial_t u \in \Lt^q(w;\Lx^p)$ in conjunction with the estimate $$\|\partial_t u \|_{\Lt^q(w;\Lx^p)} \lesssim \| f\|_{\Lt^q(w;\Lx^p)},$$ that is to say, problem~\eqref{Eq: non-autonomous equation} admits weighted maximal regularity.
	Implicit constants only depend on the parameters fixed in Agreement~\ref{Agreement: fix constants} below.
\end{theorem}
\begin{agreement}
	\label{Agreement: fix constants}
	Throughout this article, we consider the numbers $\Lambda$ and $\lambda$, as well as the numbers $\alpha$, $\beta$, and $\eps$ from Theorem~\ref{Thm: main result}, as fixed. Moreover, we reserve the symbol $\M$ for the $\Cont_t^{\beta + \eps}(\Hx^{\alpha + \eps, \nicefrac{d}{\alpha}})$ respectively $\Cont_t^{\beta + \eps}(\Contx^{\alpha + \eps})$-norm of $A$. We refer to the numbers $d$ and $m$ as dimensions, and they are also considered fixed, likewise the integrability parameters $p,q$, and the $\A_q$-weight $w$. Estimates do not depend on $w$ itself but only on an upper bound of its $\A_q$ characteristic $[w]_{\A_q}$, see Definition~\ref{Def: Muckenhoupt weight}.
\end{agreement}

Before we come to a comparison of our main result with the literature, we would like to comment on non-trivial initial values in the following remark first.

\begin{remark}
	By linearity, a non-trivial initial value $u_0$ can be included if we solve the initial value problem
	\begin{align}
		\partial_t u - \div_x A(t,x) \nabla_x u &= 0, \qquad \mathrm{in} \quad (0, T) \times \R^d, \\
		u(0) &= u_0.
	\end{align}
	When $w = 1$, then
	by the perturbation argument presented in~\cite{AO19}, which is applicable only using the regularity condition $A \in \Contt^\eps(\Lx^\infty)$, the above initial value problem is solvable provided $u_0 \in (\Lx^p, \dom(L_0))_{1-\frac{1}{q}, q}$, where $(\cdot,\cdot)_{\theta,r}$ denotes the $(\theta,r)$-real interpolation space and $L_0 \coloneqq -\div_x A(0,\cdot) \nabla_x$. If $w$ is a power weight, see Example~\ref{ex:power weight}, a similar statement can be formulated. We refrain from giving more details on this matter since it is not related to our mixed regularity condition in space and time.
\end{remark}

In the unweighted case, Fackler~\cite{Fackler18} has shown maximal regularity if $A$ is uniformly in $\mathrm{VMO}_x$ and satisfies in addition the regularity condition
\begin{align}
	A \in \begin{cases}
		\Wt^{\nicefrac{1}{2} + \eps,2}(\Lx^\infty), & \text{if } p \leq 2, \\
		\Wt^{\nicefrac{1}{2} + \eps,p}(\Lx^\infty), & \text{else}.
	\end{cases}
\end{align}
His condition is essentially the borderline case $\alpha=0$ and $\beta = \nicefrac{1}{2}$ of the parabolic relation in Theorem~\ref{Thm: main result} when $p\leq 2$. The reason why -- in contrast to Fackler -- we have to work with a H\"{o}lder condition in time will be the presence of the weight (see for instance Lemma~\ref{Lem: weighted convolution}). In the other borderline case $\alpha = 1$ and $\beta = 0$ the domains of the elliptic operators are independent of time. Consequently, maximal regularity follows from perturbation techniques~\cite{PS01}. In this sense, our regularity condition interpolates between previously known sufficient conditions, and extends these results to the time-weighted setting. Weights in time are interesting for non-linear equations with rough initial values~\cite{Critical-Spaces}.

With the same parabolic relation, the unweighted and Hilbertian case on $\R^d$ was treated by Dier and Zacher~\cite{DZ17}. Our spatial regularity condition always coincides with their hypothesis. Using Fackler's bootstrapping argument from~\cite[Thm.~6.4 \& Prop.~5.1]{Fackler18} we should be able to match their temporal regularity hypothesis. Consequently, our approach would recover their unweighted result and extend it to the non-Hilbertian setting.

\subsection{Roadmap}
\label{Subsec: Roadmap}

In this roadmap, we intend to give the reader an extensive overview of our strategy. Our proof follows a classical approach due to Acquistapace and Terreni, but incorporates an a priori improvement of weak solutions in the spatial variable using a commutator argument.

The starting point is a weak solution theory for the generalized problem~\eqref{Eq: shifted non-autonomous equation}. This generalization permits us to use an approximation argument later on. Classically, this is due to Lions in the Hilbertian situation. Fackler used the result of Pr\"uss and Schnaubelt~\cite{PS01} to have a $(p,q)$-version of Lions' result at hand. We cannot do this, as~\cite{PS01} does not yield implied constants that are uniform in the coefficients. However, we will need such a control for the a priori improvement of weak solutions in the spatial variable. We will come back to this at the very end of this roadmap. Hence, instead, we employ a framework of Dong and Kim~\cite{Dong-Kim-TAMS} to treat complex systems in divergence form over spaces of the type $\Lt^q(w;\Wx^{-1,p})$. Another advantage of the result of Dong and Kim are weighted estimates in time for weak solutions. This will be done in Section~\ref{Sec: existence weak solutions}, and consists of relating their notions with ours, as well as verifying an oscillation condition.

As is classical in the Acquistapace--Terreni approach, we derive a representation formula for weak $(p,q)$-solution in Section~\ref{Subsec: AT formula}. Fix $\tt \in (0,T)$. The formula reads
\begin{align}
	u(\tt) &= \int_0^\tt \e^{-(\tt-s) (\cB_\tt + \shift)} \bigl( \cB_\tt - \cB_s \bigr) u(s) \d s + \int_0^\tt \e^{-(\tt-s) (\cB_\tt + \shift)} f(s) \d s,
\end{align}
where the operator $\cB_\tt + \shift$ replaces the operator $\cL_\tt$ when passing from~\eqref{Eq: non-autonomous equation} to~\eqref{Eq: shifted non-autonomous equation} with regularized coefficients.
For maximal regularity, we have to estimate the term $(\cB_\tt + \shift) u(\tt)$. Formally, this leads to the operators
\begin{align}
	S_1(u)(\tt) \mapsto &\int_0^\tt (\cB_\tt + \shift) \e^{-(\tt-s) (\cB_\tt + \shift)} (\cB_\tt - \cB_s) u(s) \d s, \\
	S_2(f)(\tt) \mapsto &(\cB_\tt + \shift) \int_0^\tt \e^{-(\tt-s) (\cB_\tt + \shift)} f(s) \d s.
\end{align}
The commutation between $(\cB_\tt + \shift)$ and the integral in $S_1$ will be justified during the proof of our main result. Consequently, to establish maximal regularity, we have to bound the operators $S_1$ and $S_2$. This is the topic of Section~\ref{Sec: estimates solution formula}. Observe, however, that the operator $S_2$ acts on the data $f$, but $S_1$ acts on the weak $(p,q)$-solution $u$. This has the following effect: for $S_2$, we plainly desire to show $\Lt^q(w;\Lx^p)$-bounds. These will follow from a weighted and operator-valued pseudo-differential operator result. For $S_1$, however, the target space is still $\Lt^q(w;\Lx^p)$, but higher regularity of weak solutions lets us vary the norm of the data space. To be more precise, in the classical approach as employed by Fackler~\cite{Fackler18}, the data space is $\Lt^q(w;\Wx^{1,p})$. The fundamental gain in our approach is that we will replace that data space by the space $\Lt^q(w;\Wx^{1+\alpha,p})$. This has the effect that less restrictive kernel bounds for $S_1$ compared to~\cite{Fackler18} suffice. We give more details on this in a moment.

Let us come back to the operator $S_2$. The classical approach is to rewrite this operator as a pseudo-differential operator. This will be presented in Section~\ref{Subsec: S2 bdd}. To do so, we have to restrict to a class of more regular right-hand sides $f$. This is, however, not a restriction, since we can use a standard approximation argument for the equation. This will be explained in Step~1 in the proof of Theorem~\ref{Thm: main result} in Section~\ref{Sec: proof main result}. We emphasize that this approximation argument does not rely, yet, on the explicit control of implicit constants for weak $(p,q)$-solutions. Eventually,~\cite{PS06} leads to boundedness of $S_2$ provided we can verify that $(\tau, s) \mapsto 2\pi i \tau (2\pi i \tau + (\cB_s + \shift))^{-1}$ satisfies some $R$-boundedness and regularity conditions. The precise assumption and its verification are presented in Lemma~\ref{Lem: R-Yamazaki}. This uses two ingredients. First, that the coefficients are $\Contt^\eps(\Lx^\infty)$. Second, that the operators $(\cB_\tt + \shift)$ are jointly $R$-sectorial. Let us remark that the results in~\cite{PS06} are not weighted, but we will explain the necessary changes.

Uniform $R$-sectoriality is treated in Section~\ref{Subsec: uniform R-bounds}. On the one hand, we have to carefully trace the constants in well-known results on $R$-boundedness (more precisely, the approach based on off-diagonal bounds from~\cite{KW}). On the other hand, we combine the elliptic solvability theory of Dong and Kim (see Proposition~\ref{Prop: DK elliptic}) with recent advances around the Kato square root property~\cite{Lp-Kato} to eventually prove $\Lx^p$-boundedness for the semigroup generated by $-(\cB_\tt + \shift)$ with uniform constants in Theorem~\ref{Thm: Lp bounds semigroup}. This result is complemented by further insights on elliptic operators with minimal spatial regularity in Section~\ref{Sec: Elliptic}. In contrast to~\cite{Fackler18}, we are able to also treat complex systems. This is because we do not rely on the Gaussian bounds from~\cite{Auscher-Tchamitchian} anymore.

We come back to the operator $S_1$. As already mentioned, the plan is to show the boundedness $$S_1 \colon \Lt^q(w;\Wx^{1+\alpha,p})\to \Lt^q(w;\Lx^p).$$ This will turn out to be sufficient owing to the a priori estimate $\| u \|_{\Lt^q(w;\Wx^{1+\alpha,p})} \lesssim \| f \|_{\Lt^q(w;\Lx^p)}$ for weak solutions -- this is the higher spatial regularity that was already alluded before. The (weighted) boundedness for $S_1$ follows from a good bound of convolution type for its integral kernel (this is the reason for the Hölder condition in time), and Lemma~\ref{Lem: weighted convolution}. The kernel bound is established in Lemmas~\ref{Lem: estimate kernel of S1 part 1} and~\ref{Lem: estimate kernel of S1 part 2}. Lemma~\ref{Lem: estimate kernel of S1 part 1} is in some sense the central ingredient of this paper, as it is the only result that uses the full mixed regularity in time and space. There, we use the spatial regularity of our coefficients to have $\Wx^{\alpha,p}$-multipliers at our disposal (Lemma~\ref{Lem: multiplier}), which eventually leads to estimates against $\Wx^{1+\alpha,p}$. The spatial Sobolev condition for the coefficients is optimal (up to an $\eps$) for this multiplier result.

The missing piece is the higher spatial regularity of weak solutions, the subject of Section~\ref{Sec: higher regularity}. Recall for this that the $\Wx^{1+\alpha,p}$-norm can be given by $\| \cdot \|_{\Lx^p} + \| \partial^\alpha_x \cdot \|_{\Wx^{1,p}}$, where $\partial^\alpha_x$ is the fractional derivative of order $\alpha$. Our plan is to control the latter term by showing that $\partial^\alpha_x u(t,x)$ is a weak $(p,q)$-solution for some admissible right-hand side. Formally, one has
\begin{align}
\label{Eq: formal commutator argument intro}
	\partial_t (\partial^\alpha_x u) - \div_x B(t,x) \nabla_x (\partial^\alpha_x u) + \shift (\partial^\alpha_x u) = \partial^\alpha_x f - \div_x [B(t,\cdot), \partial^\alpha_x] \nabla_x u.
\end{align}
Then, the right-hand side is in $\Lt^q(w;\Wx^{-1,p})$ if the commutator $$[B(t,\cdot), \partial^\alpha_x] \coloneqq B(t,\cdot) \partial^\alpha_x - \partial^\alpha_x B(t,\cdot)$$ is $\Lt^q(w;\Lx^p)$-bounded (up to some absorption term in the case of $\Hx^{\alpha+\eps,\nicefrac{d}{\alpha}}$ coefficients). Owing to the spatial regularity of the coefficients, the latter fact is true according to Lemma~\ref{Lem: commutator estimate}. Nevertheless, there remain some technical difficulties. In the first place, $u$ is only in $\Lt^q(w;\Wx^{1,p})$, so neither can we plug $\partial^\alpha_x u$ into the equation, nor can we justify the necessary calculations to show~\eqref{Eq: formal commutator argument intro}. The way out are an approximation argument in which we use regularized coefficients in conjunction with the difference quotient method (see Steps~1 and~2 in the proof of Proposition~\ref{Prop: higher regularity}), and the fact that on the whole space $\partial_x^\alpha$ and $\nabla$ commute. Note that this step also excludes spatial weights, since then the norm would not be translation invariant anymore. Afterwards, when we want to take the limit in order to get back to our original equation, it is crucial to have control over the implied constants in the weak $(p,q)$-solution theory from Theorem~\ref{Thm: weak solutions} in terms of the coefficients.

\subsection*{Notation} The finite time $T > 0$ and dimension $d \geq 1$ as well as system size $m$ were already fixed in the introduction. The variables $x$ and $t$ are supposed to be quantified over $\R^d$ and $(0,T)$, respectively. By $\tt$ we indicate a fixed (but arbitrary) number in $(0,T)$. For $\varphi \in (0,\pi)$ write $\Sec_\varphi \coloneqq \{ z \in \C \setminus \{ 0 \} \colon |\arg(z)| < \varphi \}$ for the open sector of opening angle $\varphi$ around the positive real axis; also put $\Sec_0 \coloneqq (0,\infty)$. Write $\z \colon z \mapsto z$ for the identity map. It will be clear from the context on which set $\z$ is defined, usually on an open sector. If $T$ is an operator admitting a functional calculus, we write $f(T)$ or $[f](T)$ for the operator $T$ plugged into the function $f$ via its functional calculus. Often, $f$ is defined by an expression that involves the function $\z$, for instance $f = \z (1 + \z)^{-1}$.

\section*{Acknowledgments}

The author was partially supported by the Studienstiftung des deutschen Volkes, the ANR project RAGE: ANR-18-CE-0012-01, and the Humboldt foundation.
The author thanks Moritz Egert for hospitality and valuable discussions on the topic during a stay in Orsay in 2019.
The author thanks Fabian Gabel and Hannes Meinlschmidt for discussions on the topic.
Finally, the author thanks the anonymous referee for their remarks.

\section{Function spaces and weights}
\label{Sec: spaces and weights}

In this section, we review some facts from function space theory and Muckenhoupt weights, thereby introducing also our notation and some further conventions. However,
we assume that the reader is familiar with standard function space and weighted theory. For further background, the reader can, for instance, consult the monographs~\cite{Triebel} for function spaces and~\cite{RubioDeFrancia} for Muckenhoupt weights.

\subsection{Spatial smoothness spaces}

For $s\in \R$ and $p\in (1,\infty)$, write $\Hx^{s,p}$ for the Bessel potential space of order $s$ and integrability $p$. For a positive integer $k$ one has $\Hx^{k,p} = \Wx^{k,p}$. We also put $\Wx^{s,p} \coloneqq \Hx^{s,p}$. Our convention is that we use the $\Wx$-scale to denote regularity of solutions, and the $\Hx$-scale to measure regularity of coefficients. The fractional Sobolev spaces respect the usual lifting property~\cite[Sec.~2.3.4]{Triebel}. Also, the $\Lx^2$ inner product extends to a duality pairing between the spaces $\Hx^{s,p}$ and $\Hx^{-s,p'}$. Moreover, the $\Hx^{s,p}$ spaces interpolate naturally by means of the complex interpolation method due to Calder\'{o}n--Lions.

Introduce the functional
\begin{align}
	S^\alpha f(x) \coloneqq \Bigl( \int_0^\infty \Bigl( \int_{|y| \leq 1} |f(x+ry) - f(x)| \d y \Bigr)^2 \frac{\d r}{r^{1+2\alpha}} \Bigr)^\frac{1}{2}.
\end{align}
If $0 < \alpha < 1$, then the space $\Hx^{\alpha,p}$ consists of all $f\in \Lx^p$ such that $S^\alpha f \in \Lx^p$, and $f\mapsto \| f \|_{\Lx^p} + \| S^\alpha f \|_{\Lx^p}$ defines an equivalent norm on $\Hx^{\alpha,p}$, see~\cite[Thm.~2.3]{Strichartz}. This leads to the following multiplier result.
\begin{lemma}[Multiplier on fractional Sobolev spaces]
\label{Lem: multiplier}
	Let $p\in (1,\infty)$, $0<\alpha<1$, and $\eps > 0$. Let $X = \Hx^{\alpha+\eps,\nicefrac{d}{\alpha}}$ if $p < \nicefrac{d}{\alpha}$ and $X = \Contx^{\alpha+\eps}$ otherwise. Then functions in $X$ are multipliers on $\Wx^{\alpha,p}$ and one has the estimate
	\begin{align}
		\| mf \|_{\Wx^{\alpha,p}} \lesssim \| m \|_X \| f \|_{\Wx^{\alpha,p}},
	\end{align}
	where the implicit constant depends on $\alpha$, $p$, $\eps$, and dimension.
\end{lemma}
\begin{proof}
	We appeal to the aforementioned characterization. First, $\| mf \|_{\Lx^p} \leq \| m \|_{\Lx^\infty} \| f \|_{\Lx^p}$, and $\| m \|_{\Lx^\infty} \lesssim \| m \|_X$ is clear when $X$ is a H\"older space, and follows from the (fractional) Sobolev embedding theorem when $X$ is a Sobolev space.

	Next, an expansion of $S^\alpha (mf)(x)$ and the triangle inequality show $$S^\alpha (mf)(x) \leq \| m \|_{\Lx^\infty} S^\alpha f(x) + |f(x)| S^\alpha m(x),$$ compare with~\cite[Thm.~2.1]{Strichartz}. The first term can be estimated with the arguments from the beginning of the proof, this time using $S^\alpha f \in \Lx^p$. For the second term, we distinguish cases for $X$.

	\textbf{Case 1}: $X=\Contx^{\alpha+\eps}$.
	In the definition of $S^\alpha m(x)$, we split the integral in $r$ at height $1$. If $r \leq 1$, we use the H\"older regularity of $m$, to estimate this part by a constant (independent of $x$). Similarly, when $r \geq 1$, we use boundedness of $m$. In summary, $S^\alpha m(x)$ is bounded by a constant depending linearly on $\| m \|_{\Contx^{\alpha+\eps}}$, which concludes this case.

	\textbf{Case 2}: $X=\Hx^{\alpha+\eps,\nicefrac{d}{\alpha}}$.
	With the relation $\nicefrac{1}{p} - \nicefrac{\alpha}{d} \eqqcolon \nicefrac{1}{q}$ (observe that $q$ is finite by hypothesis on $p$), we use H\"older's inequality to give $\| f S^\alpha m \|_{\Lx^p} \leq \| f \|_{\Lx^q} \| S^\alpha m \|_{\Lx^{\nicefrac{d}{\alpha}}}$. By choice of $q$, one has the Sobolev embedding $\| f \|_{\Lx^q} \lesssim \| f \|_{\Hx^{\alpha, p}}$, which concludes the proof.
\end{proof}

\begin{definition}
\label{Def: fractional derivative}
	The operator $\partial^\alpha_x$ is defined as the (unbounded) Fourier multiplication operator on $\Lx^2$ with symbol $|\xi|^\alpha$. It extrapolates\footnote{Here, this means that $\partial^\alpha_x$ extends from $\Wx^{\alpha,p} \cap \Wx^{\alpha,2}$ to a bounded operator $\Wx^{\alpha,p} \to \Lx^p$ by continuity.} to a bounded operator $\Wx^{\alpha,p} \to \Lx^p$ and we keep writing $\partial^\alpha_x$.
\end{definition}

The mapping $f\mapsto \| f \|_{\Lx^p} + \| \partial_x^\alpha \|_{\Lx^p}$ yields another equivalent norm on $\Wx^{\alpha,p}$, see ~\cite[p.~133]{Stein}.

Sometimes, we also use the Besov spaces $\B^{s}_{p,p}$ with $s\geq 0$. They consist of all functions $f$ in $\L^p$ such that the norm
\begin{align}
	\| f \|_{\B^{s}_{p,p}} \coloneqq \| f \|_{\Lx^p} + \Bigl( \int_{\R^d} \int_{\R^d} \left| \frac{f(y)-f(x)}{|y-x|^{s}} \right|^p \frac{\d y \d x}{|y-x|^d} \Bigr)^\frac{1}{p}
\end{align}
is finite. By real interpolation, one has for $t > s \geq 0$ the continuous inclusion $\Hx^{t,p} \subseteq \B^{s}_{p,p}$.

\subsection{Muckenhoupt weights and parabolic spaces}

\begin{definition}[Muckenhoupt weights]
\label{Def: Muckenhoupt weight}
	Let $q \in (1, \infty)$. A locally integrable function $w \colon \R \to [0, \infty)$ is a \emph{Muckenhoupt weight} for $q$, write $w\in \A_q$, if the quantity
	\begin{align}
		[w]_{\A_q} \coloneqq \sup_I \left( \frac{1}{|I|} \int_I w \d x \right) \left( \frac{1}{|I|} \int_I w^{-\frac{1}{q-1}} \d x \right)^{q-1}
	\end{align}
	is finite, where the supremum is taken over all intervals $I \subseteq \R$. If $q$ is clear from the context, define the \emph{dual weight} to $w$ by $w' \coloneqq w^{-\frac{1}{q-1}}$.
\end{definition}

\begin{example}[Power weights]
	\label{ex:power weight}
	Let $q\in (1, \infty)$ and $-1 < \kappa < q-1$. Consider the weight $w(t) = t^\kappa$. Then $w\in \A_q$. Weights of this type are called \emph{power weights}. Such weights are prototypical for the application of our theory in non-linear problems.
\end{example}

Let $X$ be a spatial smoothness space, $q\in (1,\infty)$ and $w\in \A_q$. We consider the weighted parabolic spaces $\Lt^q(w;X) \coloneqq \L^q(0,T, w; X)$ and $\Wt^{1,q}(w;X)$, where the latter space consists of all $u \in \Lt^q(w;X)$ with $\partial_t u$ again in $\Lt^q(w;X)$. Note that functions in $\Lt^q(w;X)$ are locally integrable by the $\A_q$-condition, hence the distributional derivative is well-defined.

If $Y \subseteq X$ is dense, then $\Cont_0^\infty(\R; Y)$ is dense in $\L^q(\R,w;X)$. One has the usual duality relation $(\Lt^{q'}(w';X^*))^* = \Lt^{q}(w;X)$, which extends the pairing between unweighted spaces. The following well-known lemma is a handy substitute for Young's convolution inequality in the context of weighted spaces (here, $X=\C$).

\begin{lemma}
\label{Lem: weighted convolution}
	Let $k \colon \R \to [0, \infty)$ be measurable, radial, decreasing and integrable. Then
	\begin{align}
		|(k \ast f)(x)| \lesssim \| k \|_1 \MO f(x) \qquad (f\in \Lloc^1),
	\end{align}
	where $\MO$ is the maximal operator. In particular, if $q\in (1, \infty)$ and $w\in \A_q$, one has the weighted estimate
	\begin{align}
		\| k \ast f \|_{\Lt^q(\R, w)} \lesssim \| k \|_1 \| f \|_{\Lt^q(\R, w)} \qquad (f\in \Lt^q(\R, w)).
	\end{align}
\end{lemma}

\section{Uniform estimates for elliptic operators}
\label{Sec: Elliptic}

In Section~\ref{Sec: Introduction} we have introduced the elliptic operators $\{ \cL_t\}_{0<t<T}$. We will associate parts in $\Lx^2$ with these operators, and show \emph{uniform} bounds for their associated semigroups and square roots. We will also transfer semigroup bounds to the space $\Wx^{-1,p}$. The cornerstone for the results in this section is the well-posedness result for parabolic systems in divergence form due to Dong and Kim~\cite{Dong-Kim-TAMS}.

\subsection{Elliptic coefficients}

We stay slightly more general here, which will become handy for technical reasons later on, for instance in Section~\ref{Sec: higher regularity}.
That being said, we introduce the following class of regular elliptic coefficients, which includes the coefficients of the non-autonomous problems studied in this article.

\begin{definition}
\label{Def: coefficient class}
	Let $\gamma > 0$ and $N\geq 0$. Denote by $\E(\Lambda, \lambda, \gamma, N)$ the class of \emph{elliptic coefficients} with coefficient bounds $\Lambda$ and $\lambda$ that are $\Cont^\gamma$ with norm at most $N$. More precisely, this class consists of all functions $B\colon \R^d \to \C^{dm \times dm}$ which satisfy
	\begin{align}
		|B(x)| \leq \Lambda \quad \& \quad \sum_{\sysalpha,\sysbeta = 1}^m \Re (B(x)^{\sysalpha \sysbeta} \xi^\sysalpha \SP \xi^\sysbeta) \geq \lambda |\xi|^2 \qquad (\xi \in \C^{dm}),
	\end{align}
	and the regularity condition
	\begin{align}
		\frac{|B(x+h)-B(x)|}{|h|^\gamma} \leq N \qquad (h \in \R^d \setminus \{ 0 \}).
	\end{align}
\end{definition}

\begin{remark}
\label{Rem: coefficient class}
	Note that $A(\tt, \cdot) \in \E(\Lambda, \lambda, \eps, \M)$. In the case $p < \nicefrac{d}{\alpha}$, this follows from embedding results for smoothness spaces, see~\cite[Thm.~2.8.1.~(e)]{Triebel}.
\end{remark}

\subsection{Elliptic systems and weak \texorpdfstring{\boldsymbol{$(p,q)$}}{(p,q)}-solutions}

We associate with a coefficient function $B$ a form and an operator $\Wx^{1,2} \to \Wx^{-1,2}$.

\begin{definition}
\label{Def: b and B}
	Let $B\in \E(\Lambda, \lambda, \gamma, N)$. Define the form
	\begin{align}
		b \colon \Wx^{1,2} \times \Wx^{1,2} \to \C, \qquad b(u,v) = \int_{\R^d} B(x) \nabla_x u(x) \cdot \overline{\nabla_x v(x)} \d x,
	\end{align}
	and associate with it the operator
	\begin{align}
		\cB \colon \Wx^{1,2}\to \Wx^{-1,2} \quad \text{via} \quad \langle \cB u, v \rangle_{\Wx^{-1,2},\Wx^{1,2}} = b(u,v)  \qquad (u,v\in \Wx^{1,2}).
	\end{align}
	The form $b$ is likewise bounded on $\Wx^{1,p} \times \Wx^{1,p'}$, so that $\cB$ is also a bounded operator $\Wx^{1,p} \to \Wx^{-1,p}$. We do not distinguish these objects notation-wise.
\end{definition}

Given a family $\{ \cB_t \}_{0<t<T}$ induced by coefficients $B(t, \cdot) \in \E(\Lambda, \lambda, \gamma, N)$ and a parameter $\shift \in \R$, associate with them the non-autonomous evolution problem
\begin{align}
\label{Eq: shifted non-autonomous equation}
\tag{P'}
	\partial_t u(t) + \cB_t u(t) + \shift u(t) = f(t), \qquad u(0) = 0.
\end{align}
The following definition makes precise what we understand under a solution to~\eqref{Eq: shifted non-autonomous equation}. With the choices $\cB_t = \cL_t$ and $\shift=0$, this clarifies in particular the solution concept for the problem~\eqref{Eq: non-autonomous equation} from the introduction.
\begin{definition}
\label{Def: weak solution}
	Given $f\in \Lt^q(w;\Wx^{-1,p})$, $p,q \in (1,\infty)$, and $\kappa \in \R$, call a function
	$u\in \Lt^q(w;\Wx^{1,p})$
	a \emph{weak $(p,q)$-solution} of~\eqref{Eq: shifted non-autonomous equation},
	if $u(0)=0$,
	and if the integral equation
	\begin{align}\label{Eq: weak pq-solution}
	\tag{IE}
	\begin{split}
		&\int_0^T -\varphi'(s) ( u(s) \SP g ) + \varphi(s) b_s(u(s), g)
		+ \shift \varphi(s)(u(s) \SP g) \d s \\
		&\qquad\qquad
		= \int_0^T \varphi(s) \langle f(s), g \rangle_{\Wx^{-1,p}, \Wx^{1,p'}} \d s
	\end{split}
	\end{align}
	holds for all $\varphi \in \Cont_0^\infty(0,T; \C)$ and $g\in \Cont_0^\infty(\R^d; \C)$.
\end{definition}

\begin{remark}
\label{Rem: weak solution}
We give some more clarifying comments regarding Definition~\ref{Def: weak solution}.
\begin{enumerate}
	\item Functions in $\Lt^q(w)$ with $w\in \A_q$ are locally integrable, hence the pairings in~\eqref{Eq: weak pq-solution} are well-defined.
	\item It follows from duality that a weak $(p,q)$-solution $u$ of~\eqref{Eq: shifted non-autonomous equation} has a weak derivative $\partial_t u$ in $\Lt^q(w;\Wx^{-1,p})$ that coincides with $f(t)-\cB(t)u(t) - \shift u(t)$ for almost all $t$. \label{Item: weak derivative}
	\item\label{it: continuous} A weak $(p,q)$-solution is continuous at $0$ with values in $\Wx^{-1,p}$, which renders the initial condition meaningful. For the weighted case, this is presented in~\cite[Lem.~4.1]{GV17a}.
	\item Existence and uniqueness of weak $(p,q)$-solutions are independent of the parameter $\shift$. Indeed, if $u$ is a weak $(p,q)$-solution to the parameter $\shift$, then $v(t) = \e^{\shift t} u(t)$ is a weak $(p,q)$-solution to the right-hand side $\e^{\shift s} f$ with $\shift=0$, and vice versa. Note that this imports a dependence on $T$ for the implicit constants. \label{Item: shift equation}
	\item The integral equation~\eqref{Eq: weak pq-solution} extends to $g\in \Wx^{1,p'}$ by continuity. \label{Item: more general g}
\end{enumerate}
\end{remark}

The parameter $\shift$ is supposed to be taken sufficiently large (in particular, we tacitly assume $\shift \geq 1$). This is quantified by the results in~\cite{Dong-Kim-TAMS}. In particular, we can ensure ellipticity in this way. We emphasize that the choice of $\shift$ can be made uniform in the quantities mentioned in Agreement~\ref{Agreement: fix constants}.

Let us agree for the rest of this section that $B$ denotes any fixed coefficient function from the class $\E(\Lambda, \lambda, \eps, \M)$. Implicit constants are allowed to depend on $p$, $\Lambda$, $\lambda$, $\eps$, $M$, and dimensions.

As a consequence of ellipticity, there is some $\omega \in [0,\nicefrac{\pi}{2})$ depending on $\Lambda$, $\lambda$, and $\shift$ such that the numerical range of $b + \shift (\cdot \SP \cdot )_2$ is contained in the closed sector $\overline{\Sec}_\omega$ of opening angle $2\omega$. Furthermore, using Definition~\ref{Def: coefficient class} and the Lax--Milgram lemma, $\cB + \shift + \rho$ is invertible for all $\rho \geq 0$. In particular, $\cB + \shift$ is itself invertible as an operator $\Wx^{1,2} \to \Wx^{-1,2}$.

As a consequence of the H\"older regularity of the coefficients, $\cB + \shift$ extrapolates moreover to an isomorphism $\Wx^{1,p} \to \Wx^{-1,p}$ for \emph{all} $p\in (1,\infty)$. The argument divides into two steps. First, the autonomous problem associated with $\cB + \shift$ is well-posed according to~\cite{Dong-Kim-TAMS}. We will give further information on that result and its applicability in our context in Section~\ref{Sec: existence weak solutions}, see in particular Lemma~\ref{lem: integral inequality}.
Second, the well-posedness of the original elliptic problem together with an estimate for its solutions follow by applying a cutoff argument to a stationary solution~\cite[Proof of Thm.~2.2]{Dong-Kim-CoV}. The result can then be summarized as follows.

\begin{proposition}
\label{Prop: DK elliptic}
	Let $p\in (1,\infty)$. The operator $\cB + \shift$ extrapolates to an invertible operator $\Wx^{1,p} \to \Wx^{-1,p}$. Given $f\in \Wx^{-1,p}$, write $u \in \Wx^{1,p}$ for the unique solution to the equation $(\cB + \shift) u = f$. Then, one has the estimate $\| u \|_{\Wx^{1,p}} \lesssim \| f \|_{\Wx^{-1,p}}$.
\end{proposition}

\begin{remark}
\label{Rem: compatible}
	The solutions provided by Proposition~\ref{Prop: DK elliptic} are \emph{compatible} to Lax--Milgram solutions in the following sense. Given $f\in \Wx^{-1,p} \cap \Wx^{-1,2}$, let $u$ be the solution in $\Wx^{1,p}$ provided by Proposition~\ref{Prop: DK elliptic}, and $v$ be the solution in $\Wx^{1,2}$ provided by the Lax--Milgram lemma. Then $u$ and $v$ coincide. Indeed, this is a consequence of local compatibility in complex interpolation scales~\cite[Thm.~8.1]{KMM} and the fact that Proposition~\ref{Prop: DK elliptic} provides a solution for \emph{all} $p\in (1,\infty)$.
\end{remark}

\begin{remark}
\label{Rem: DK works with half shift}
        The result in~\cite{Dong-Kim-TAMS} only requires that $\shift$ is larger than a certain threshold quantified by the parameters fixed in Agreement~\ref{Agreement: fix constants}. Hence, to ensure that all results in Section~\ref{Sec: Elliptic} remain true when $\shift$ is replaced by $\nicefrac{\kappa}{2}$, we pick $\shift$ a bit larger for good measure. We will exploit this observation in Section~\ref{Sec: estimates solution formula}.
\end{remark}

\subsection{The elliptic operator on \texorpdfstring{\boldsymbol{$\Lx^2$}}{L2} and mapping properties}
\label{Subsec: elliptic operators}

In virtue of the embedding $\Lx^2 \subseteq \Wx^{-1,2}$,
define the part of $\cB$ in $\Lx^2$ and denote it as an abuse of notation also by the symbol $B$ (it will be clear from the context if $B$ denotes the coefficient function or the part in $\Lx^2$). Of course, the part of $\cB + \shift$ in $\Lx^2$ coincides with $B+\shift$.
One has that $B+\shift$ is a densely defined, invertible, and m-$\omega$-sectorial operator in $\Lx^2$ with domain $\dom(B+\shift)=\dom(B)$. In particular, $-(B+\shift)$ generates a holomorphic semigroup of contractions $\{\e^{-z (B+\shift)} \}_{z\in \Sec_{\nicefrac{\pi}{2}-\omega}}$ on $\Lx^2$. We will tacitly employ some properties of the sectorial functional calculus of $B+\shift$. The reader can consult~\cite[Chap.~7]{Haase} for further background.

Owing to~\cite[Lem.~7.3]{Lp-Kato}, we deduce $\Lx^p$-bounds for the semigroup generated by $-(B+\shift)$ as a consequence of Proposition~\ref{Prop: DK elliptic} and Remark~\ref{Rem: compatible}.

\begin{theorem}
\label{Thm: Lp bounds semigroup}
	Let $p\in (1,\infty)$ and $\varphi \in [0, \nicefrac{\pi}{2}-\omega)$. One has the estimate
	\begin{align}
		\| \e^{-z (B+\shift)} f \|_{\Lx^p} \lesssim \| f \|_{\Lx^p} \qquad (z\in \Sec_\varphi, f\in \Lx^p \cap \Lx^2).
	\end{align}
\end{theorem}

\begin{remark}
	In~\cite[Lem.~7.3]{Lp-Kato}, only the case $p \geq 2$ is presented. The case $p \leq 2$ either follows by a duality argument with $\cB^* + \shift$, or by repeating the calculation in~\cite{Lp-Kato}, but changing the order in which $\H^\infty$-calculus and $(\cB+\shift)^{-1}$ are applied.
\end{remark}

\subsection{Square roots and bounds on \texorpdfstring{\boldsymbol{$\Wx^{-1,p}$}}{W-1p}}
\label{Subsec: square roots and semigroup}

As an m-$\omega$-sectorial operator, $B+\shift$ possesses a square root $(B+\shift)^\frac{1}{2}$. It acts as an isomorphism $\Wx^{1,2} \to \Lx^2$ according to the solution of the Kato square root problem~\cite{Kato}. As a consequence of coefficient regularity, $(B+\shift)^\frac{1}{2}$ extrapolates to an isomorphism $\Wx^{1,p} \to \Lx^p$ for \emph{all} $p\in (1,\infty)$.
Similar ideas were already employed in~\cite{Fackler18}, but relying on the Gaussian property, which was only established in the scalar case $m=1$ and is notably more technical. Instead, we use recent results established by the author in~\cite[Thm.~1.1]{Lp-Kato}.
Indeed, in the case $p\leq 2$, its application is justified by Theorem~\ref{Thm: Lp bounds semigroup}, whereas in the case $p\geq 2$,
we appeal to Proposition~\ref{Prop: DK elliptic} in conjunction with Remark~\ref{Rem: compatible}.

\begin{theorem}
\label{Thm: square root Lp bounds}
	Let $p\in (1,\infty)$. Then $(B+\shift)^\frac{1}{2}$ extrapolates to a (compatible) isomorphism $\Wx^{1,p} \to \Lx^p$.
\end{theorem}

Theorem~\ref{Thm: square root Lp bounds} allows us to translate the $\Lx^p$-bounds for $\{ \e^{-z (B+\shift)} \}_{z\in \Sec_\varphi}$ from Theorem~\ref{Thm: Lp bounds semigroup} to $\Wx^{-1,p}$-bounds.
\begin{proposition}
\label{Prop: semigroup W-1p bounds}
	Let $p\in (1,\infty)$ and $\varphi \in [0,\nicefrac{\pi}{2}-\omega)$. One has the estimate
	\begin{align}
		\| \e^{-z (B+\shift)} f \|_{\Wx^{-1,p}} \lesssim \| f \|_{\Wx^{-1,p}} \qquad (z\in \Sec_\varphi, f\in \Wx^{-1,p} \cap\Lx^2).
	\end{align}
	In particular, $\{ \e^{-z (B+\shift)} \}_{z\in \Sec_\varphi}$ extrapolates to a semigroup on $\Wx^{-1,p}$ with generator $-(\cB+\shift)$.
\end{proposition}

\begin{proof}
	Let $z\in \Sec_\varphi$ and $f\in \Wx^{-1,p} \cap\Lx^2$. As a primer, let us show
	\begin{align}
	\label{Eq: square root for negative regularity order}
		\| (B+\shift)^{-\frac{1}{2}} f \|_p \lesssim \| f \|_{\Wx^{-1,p}}.
	\end{align}
	We employ a duality argument. To this end, let $h\in \Lx^{p'} \cap \Lx^2$. Note that the coefficient class $\E(\Lambda, \lambda, \eps, M)$ is invariant under taking adjoints. Calculate using Kato's square root property and Theorem~\ref{Thm: square root Lp bounds} (applied with $B^*$ and $p'$ instead of $B$ and $p$) that
	\begin{align}
	\label{Eq: square root dual estimate}
		|( (B+\shift)^{-\frac{1}{2}} f \SP h)| &=  |( f \SP (B^*+\shift)^{-\frac{1}{2}} h )| \\
		&\leq \| f \|_{\Wx^{-1,p}} \| (B^*+\shift)^{-\frac{1}{2}} h \|_{\Wx^{1,p'}} \\
		&\lesssim \| f \|_{\Wx^{-1,p}} \| h \|_{p'}.
	\end{align}
	Duality lets us conclude this first claim.

	Next, write $$\e^{-z (B+\shift)} f = \e^{-z (B+\shift)} (B+\shift)^\frac{1}{2} (B+\shift)^{-\frac{1}{2}} f = (B+\shift)^\frac{1}{2} \e^{-z (B+\shift)} (B+\shift)^{-\frac{1}{2}} f.$$ Let $g\in \Wx^{1,p'} \cap \Lx^2$, and calculate similarly as above, but using furthermore Theorem~\ref{Thm: Lp bounds semigroup}, that
	\begin{align}
		|\langle \e^{-z (B+\shift)} f, g \rangle| &= |( \e^{-z (B+\shift)} (B+\shift)^{-\frac{1}{2}} f \SP (B^*+\shift)^\frac{1}{2} g )| \\
		&\leq \| \e^{-z (B+\shift)} (B+\shift)^{-\frac{1}{2}} f \|_p \| (B^*+\shift)^\frac{1}{2} g \|_{p'} \\
		&\lesssim \| (B+\shift)^{-\frac{1}{2}} f \|_p \| g \|_{\Wx^{1,p'}}.
	\end{align}
	Duality and~\eqref{Eq: square root for negative regularity order} lead to $\| \e^{-z (B+\shift)} f \|_{\Wx^{-1,p}} \lesssim \| (B+\shift)^{-\frac{1}{2}} f \|_p \lesssim \| f \|_{\Wx^{-1,p}}$.
\end{proof}

\subsection{Uniform \texorpdfstring{\boldsymbol{$R$}}{R}-sectoriality}
\label{Subsec: uniform R-bounds}

As a preparation for Section~\ref{Subsec: S2 bdd}, we show $R$-sectoriality for the set of operators $\{ B + \shift \SetSep B\in \cC \}$, where $\cC$ consists of all operators associated with coefficients in $\E(\Lambda, \lambda, \eps, \M)$, and where the $R$-bound only depends on the quantified parameters from Agreement~\ref{Agreement: fix constants}. For further background on $R$-boundedness and $R$-sectoriality, the reader can consult~\cite{KW}.

\begin{proposition}[$R$-sectoriality of $\cC$]
\label{Prop: R-sectoriality}
	Let $p\in (1,\infty)$ and $\varphi \in [0, \nicefrac{\pi}{2}-\omega)$. Then, the set $\{ \e^{-z (B + \shift)} \SetSep z\in \Sec_\varphi, B\in \cC \}$ satisfies the square function estimate
	\begin{align}
	\label{Eq: SFE semigroup}
		\Bigl\| \Bigl( \sum_{j=1}^k |\e^{-z_j (B_j + \shift)} f_j|^2 \Bigr)^\frac{1}{2} \Bigr\|_p \lesssim \Bigl\| \Bigl( \sum_{j=1}^k |f_j|^2 \Bigr)^\frac{1}{2} \Bigr\|_p \quad \bigl(z_j \in \Sec_\varphi, B_j \in \cC, f_j \in \Lx^p \cap \Lx^2 \bigr).
	\end{align}
	In particular, for $z\in \Sec_\varphi$ and $B\in \cC$ fixed, the operator $\e^{-z (B + \shift)}$ extends from $\Lx^p \cap \Lx^2$ to a bounded operator $S_B(z)$ on $\Lx^p$, $\{ S_B(z) \}_{z\in \Sec_\varphi}$ is a strongly continuous and analytic semigroup on $\Lx^p$, and the set $\{ S_B(z) \SetSep z\in \Sec_\varphi, B\in \cC \}$ is $R$-bounded with $R$-bound depending only on the parameters fixed in Agreement~\ref{Agreement: fix constants}.
\end{proposition}

\begin{remark}
\label{Rem: R-sectorial}
	Proposition~\ref{Prop: R-sectoriality} shows in particular that the semigroup in $\Lx^p$ is $R$-sectorial of the same angle as the semigroup on $\Lx^2$. Hence, we keep writing $\omega$ instead of, say, $\omega_R$.
\end{remark}

Before we come to the justification of Proposition~\ref{Prop: R-sectoriality}, let us record an important consequence that we will need later on in Section~\ref{Subsec: S2 bdd}.

\begin{corollary}
\label{Cor: R-bounded resolvent}
	Let $p\in (1,\infty)$ and $\psi \in [\nicefrac{\pi}{2}, \pi-\omega)$. Let $-B_p^\shift$ denote the generator of the semigroup $\{ S_B(t) \}_{t>0}$ from Proposition~\ref{Prop: R-sectoriality}. Then the set $\{ z (z+B_p^\shift)^{-1} \SetSep z\in \Sec_\psi, B\in \cC \}$ of operators on $\Lx^p$ is $R$-bounded, and the $R$-bound depends only on the quantities fixed in Agreement~\ref{Agreement: fix constants}.
\end{corollary}
\begin{proof}
	Fix $z\in \Sec_\psi$ and $B\in \cC$. Split $\arg(z) = \varphi + \tilde \varphi$, where $|\varphi| \in [0, \psi - \nicefrac{\pi}{2})$ and $|\tilde \varphi| \in [0, \nicefrac{\pi}{2})$. The operator $(z+B_p^\shift)^{-1}$ can be represented using the Laplace transform~\cite[Prop.~3.4.1~d)]{Haase} via
	\begin{align}
		(z+B_p^\shift)^{-1} = \e^{-i \varphi} \int_0^\infty \e^{-t |z| \e^{i \tilde \varphi}} S_B(t \e^{-i \varphi}) \d t.
	\end{align}
	Then, the claim follows from~\cite[Ex.~2.15]{KW}. Indeed, they show that $z(z+B_p^\shift)^{-1}$ is contained in the strong closure of the absolute convex hull of the semigroup generated by $-B_p^\shift$. Hence, $\{ z (z+B_p^\shift)^{-1} \SetSep z\in \Sec_\psi, B\in \cC \}$ is contained in the strong closure of the absolute convex hull of $\{ S_B(z) \SetSep z\in \Sec_{\psi-\nicefrac{\pi}{2}}, B\in \cC \}$. But taking the strong closure of the absolute convex hull of a set of operators preserves $R$-boundedness with the same $R$-bound, so we conclude using Proposition~\ref{Prop: R-sectoriality}.
\end{proof}

Given $1\leq r < 2 < s \leq \infty$ such that $p\in (r,s)$, and $B \in \cC$, Proposition~\ref{Prop: R-sectoriality} is a consequence of so-called \emph{$\Lx^r \to \Lx^s$ off-diagonal estimates} for $\{ \e^{-z (B+\shift)} \}_{z\in \Sec_\varphi}$. The general approach in the context of homogeneous spaces was presented in~\cite{KW}, and for dependence of the implied constants see~\cite[Sec.~5]{D-to-N}. To be more precise, we suppose that, for some $c>0$ and for all measurable sets $E,F \subseteq \R^d$ and $z\in \Sec_\varphi$, one has the bound
\begin{align}
\label{Eq: ODE}
	\| \ind_F \e^{-z (B + \shift)} \ind_E f \|_s \lesssim |z|^{\nicefrac{d}{2s}-\nicefrac{d}{2r}} \e^{-c \frac{\dist(E,F)^2}{|z|}} \| \ind_E f \|_r \qquad (f\in \Lx^r \cap \Lx^2).
\end{align}

Inequality~\eqref{Eq: ODE} for $r=s=2$ is known under the name \emph{Gaffney estimates} and is well-known in the literature. A version of this result that carefully keeps track of the implicit constants can be found in~\cite[Prop.~3.2]{Lp-Kato}. Likewise,~\eqref{Eq: ODE} is known for $r=2$, $s\in (2,\infty)$, and with $c=0$, as a consequence of the $\Lx^p$-bounds for the semigroup provided by Theorem~\ref{Thm: Lp bounds semigroup} and \cite[Prop.~3.2~(1)]{Memoirs}. In this case, we speak of \emph{hypercontractivity} of the semigroup. Finally,~\eqref{Eq: ODE} is then a consequence of interpolation of Gaffney estimates with hypercontractivity, taking duality and composition into account.

\section{Existence and uniqueness of weak $(p,q)$-solutions}
\label{Sec: existence weak solutions}

In this section, we consider a family of operators $\{ \cB_t \}_{0<t<T}$ associated with coefficients $B(t,\cdot) \in \E(\Lambda, \lambda, \eps, M)$ that \emph{depend $\Contx^{\eps}$ on $t$}.\footnote{Say that a family $\{B_t\}_{0 < t < T} \subseteq \E(\Lambda, \lambda, \alpha, M)$ \emph{depends $\Contx^\beta$ on $t$} if $B_t \in \E(\Lambda, \lambda, \alpha, M)$ and the mapping $t \mapsto B_t$ is $\beta$-Hölder continuous with values in $\Contx^\alpha$, that is, the scalar-valued function $t \mapsto \|B_t\|_{\Contx^\alpha}$ lies in the class $\Contt^\beta$.}
The prototype for such a family of operators is the family $\{ \cL_t \}_{0<t<T}$ from Section~\ref{Sec: Introduction} (keep Remark~\ref{Rem: coefficient class} in mind).
We aim to prove the existence and uniqueness of solutions to the associated problem \eqref{Eq: shifted non-autonomous equation} in the sense of Definition~\ref{Def: weak solution}.
To do so, we recast our original problem in the framework originating from the works of Dong and Kim \cite{Dong-Kim-ARMA, Dong-Kim-JFA, Dong-Kim-CoV, Dong-Kim-TAMS}.
This includes the introduction of a global extension in time of our original problem on $\R$ as outlined in \cite[Rem.~1]{Dong-Kim-ARMA}.
Implicit constants in this section are allowed to depend on $p$, $q$, $[w]_{\A_q}$, $\Lambda$, $\lambda$,  $\alpha, \beta$, Hölder regularity, and dimensions.

We begin by extending our coefficient family $\{B_t\}_{0 < t < T}$ to all of $\R$. We extend constantly at the endpoints, that is, we set $B_t \coloneqq B_0$ for all $t < 0$ and $B_t \coloneqq B_T$ for all $t > T$.
For such $t$, we associate of course also a form $b_t$ with $B_t$.
Note that this extension does not affect the assumed Hölder regularity of the coefficients.
Furthermore, we isometrically extend the right-hand side $f \in \Lt^q(w;\Wx^{-1,p})$ outside of $(0,T)$ by zero to arrive at a function in $\L^q(\R, w; \Wx^{-1,p})$, which we denote by $F$. Also in the sequel, we will systematically denote functions on $\R$ by capital letters to better distinguish them from their local analogs.
Given the extensions of $\{B_t\}_{0 < t < T}$ and $f$, we look for solutions $U \in \L^q(\R, w; \Wx^{1,p})$ fulfilling the extended integral equation
\begin{align}\label{Eq: extended weak pq-solution}
  \tag{EIE}
  \begin{split}
  &\int_\R -\Phi'(s) ( U(s) \SP g ) + \Phi(s) b_s(U(s), g)
  + \shift \Phi(s)(U(s) \SP g) \d s \\
  &\qquad\qquad= \int_\R \Phi(s) \langle F(s), g \rangle_{\Wx^{-1,p}, \Wx^{1,p'}} \d s,
\end{split}
\end{align}
where we use test functions $\Phi \in \Cont_0^\infty(\R)$ and $g \in \Cont_0^\infty(\R^d)$.
Dong and Kim solved a similar problem in \cite{Dong-Kim-TAMS}.
They show that, for a given $F \in \mathbb{H}^{-1}_{p,q,w}(\R \times \R^d)$ with $F = F_0 + \sum_{i = 1}^d \partial_i F_i$, $F_j \in \L^q(\R, w; \Lx^p)$, there exists a solution $U \in \mathring{\mathcal{H}}^1_{p,q,w}(\R \times \R^d)$ satisfying the integral equation
\begin{align}
	\label{eq: weak Dong-Kim}
        \tag{DKIE}
	\qquad \int_\R - ( U(s) \SP\Psi'(s)) + b_s(U(s), \Psi(s)) + \shift (U(s) \SP \Psi(s) ) \d s
	= \int_\R  \langle F(s) , \Psi(s)\rangle \d s
\end{align}
for all test functions $\Psi \in \Cont_0^\infty(\R\times \R^d)$.
We explain and compare the used function spaces in the sequel of this section.
For the notion of weak solutions employed by Dong and Kim, see also~\cite[p.~896]{Dong-Kim-ARMA} and \cite[p.~3286]{Dong-Kim-JFA}.
Furthermore, solutions to \eqref{eq: weak Dong-Kim} are subject to the a priori estimate
\begin{align}
	\label{eq: apriori Dong-Kim}
	\shift \| U\|_{\L^q(\R, w; \Lx^p)} +  \sum_{i = 1}^d \shift^{\nicefrac{1}{2}} \| \partial_i U \|_{\L^q(\R, w; \Lx^p)}
	\lesssim
	\| F_0 \|_{\L^q(\R, w; \Lx^p)} + \sum_{i = 1}^d \shift^{\nicefrac{1}{2}} \| F_i \|_{\L^q(\R, w; \Lx^p)}
\end{align}
according to \cite[Thm.~7.2]{Dong-Kim-TAMS}, where the implicit constant depends on $p$, $q$, $[w]_{\A_q}$, $\Lambda$, $\lambda$,  dimension,  and the parameters $\gamma$ and $R_0$ appearing in Lemma~\ref{lem: integral inequality}.
In particular, choosing $F = 0$ in~\eqref{eq: apriori Dong-Kim} shows the uniqueness of solutions to~\eqref{eq: weak Dong-Kim}.

The rest of this section is divided into two steps:
First, we will relate the solution concepts of \eqref{eq: weak Dong-Kim} and~\eqref{Eq: extended weak pq-solution} and show that the former implies the latter. Eventually, this leads to a solution for the original problem~\eqref{Eq: shifted non-autonomous equation}.
Second, we will check the validity of the regularity assumptions on $\{B_t\}$ from~\cite[Thm.~7.2]{Dong-Kim-TAMS} to harvest the results of the first step.
At the end of the day, this will prove the following theorem.
\begin{theorem}
\label{Thm: weak solutions}
  Given $f \in \Lt^q(w;\Wx^{-1,p})$, there exists a unique weak $(p,q)$-solution $u$ to~\eqref{Eq: shifted non-autonomous equation}, and one has the estimate
	\begin{align}
	\label{Eq: MR estimate weak solution}
		\| \partial_t u \|_{\Lt^q(w;\Wx^{-1,p})}
		+  \| \nabla_x u \|_{\Lt^q(w;\Lx^{p})}
		+ \shift \| u\|_{\Lt^q(w;\Lx^p)}
		\lesssim \| f \|_{\Lt^q(w;\Wx^{-1,p})}\,.
	\end{align}
\end{theorem}

\textbf{Step 1}: Compatibility with Dong and Kim.
In order to solve \eqref{eq: weak Dong-Kim}, Dong and Kim consider right-hand sides $F$ in the spaces $\mathbb{H}^{-1}_{p,q,w}(\R \times \R^d)$.
These spaces are isomorphic to the spaces $\L^q(\R, w; \Wx^{-1,p})$ as can bee seen from a parabolic variant of \cite[Thm.~3.9]{Adams}.
This means that the admissible right-hand sides for \eqref{eq: weak Dong-Kim} and \eqref{Eq: extended weak pq-solution} coincide.
Now, \cite[Sec.~8]{Dong-Kim-TAMS} gives the existence of a solution $U$ to \eqref{eq: weak Dong-Kim} in the regularity class $\mathring{\mathcal{H}}^1_{p,q,w}(\R \times \R^d)$, which denotes the closure of $\Cont_0^\infty(\R \times \R^d)$ in the space $\mathcal{H}^1_{p,q,w}(\R \times \R^d)$.
Since we work spatially in $\R^d$, $\mathring{\mathcal{H}}^1_{p,q,w}(\R \times \R^d) = \mathcal{H}^1_{p,q,w}(\R \times \R^d)$.
A function $U \in \mathcal{H}^1_{p,q,w}(\R \times \R^d)$ is by its very definition an element of $\L^q(\R, w; \Wx^{1,p})$.
Conversely a function in $\L^q(\R, w; \Wx^{1,p})$ that satisfies~\eqref{Eq: extended weak pq-solution} is a member of $\mathcal{H}^1_{p,q,w}(\R \times \R^d)$.
For complete definitions of the above function spaces, the reader can consult \cite[p.~3284]{Dong-Kim-JFA} and \cite[Sec.~4]{Dong-Kim-TAMS}.

Comparing the classes of test functions employed in \eqref{Eq: extended weak pq-solution} and \eqref{eq: weak Dong-Kim} reveals that Dong and Kim use a larger class of test functions in their integral formulation.
In particular, this shows that a solution to \eqref{eq: weak Dong-Kim} is also a solution to \eqref{Eq: extended weak pq-solution}.
On the other hand, recall that a function $U \in \L^q(\R, w; \Wx^{1,p})$ solving~\eqref{Eq: extended weak pq-solution}
is also an admissible function for~\eqref{eq: weak Dong-Kim}.
Using the fact that the tensors $\Phi(t)g(x)$ with $\Phi \in \Cont_0^\infty(\R)$ and $g \in \Cont_0^\infty(\R^d)$ are dense in
$\L^{q'}(\R, w'; \Wx^{-1,p'})$, we deduce by continuity (compare with Remark~\ref{Rem: weak solution}~\ref{Item: more general g}) that~\eqref{Eq: extended weak pq-solution} in particular remains to hold for test functions in
$\Cont_0^\infty(\R \times \R^d)$. Hence, we get that $U$ is also a solution for~\eqref{eq: weak Dong-Kim}, and is as such again unique.

Next, we focus on the a priori estimate~\eqref{eq: apriori Dong-Kim} and its relation to the maximal regularity estimate in Theorem~\ref{Thm: weak solutions}.
Recall $\shift \geq 1$. Then, we have
\begin{align}\label{eq: a priori estimation}
	\|U \|_{\L^q(\R, w; \Wx^{1,p})}
	\lesssim
	\shift \|U \|_{\L^q(\R, w; \Lx^p)} + \sum_{i = 1}^d \shift^{1/2} \| \partial_i U \|_{\L^q(\R, w; \Lx^p)}
	\lesssim \shift \| F \|_{\L^q(\R, w; \Wx^{-1,p})}.
\end{align}

Up to now, we have only worked out the existence and uniqueness of solutions to the extended integral equation~\eqref{Eq: extended weak pq-solution}.
Hence, it remains to get back to~\eqref{Eq: shifted non-autonomous equation}.
Recall that $F = 0$ outside of the interval $(0,T)$
by construction.
Consequently, $U = 0$ on $(-\infty, 0)$ by uniqueness, hence $U(0) = 0$ by continuity (see Remark~\ref{Rem: weak solution}~\ref{it: continuous}).
Additionally, the solution $U \in \L^q(\R, w; \Wx^{1,p})$ that has been constructed via the method above gives rise to a restriction $u=U|_{(0,T)} \in \Lt^q(w;\Wx^{1,p})$.
Then $u$ satisfies $u(0) = U(0) = 0$ by continuity and solves~\eqref{Eq: weak pq-solution}.
This shows that $u$ is the unique $(p,q)$-solution of \eqref{Eq: shifted non-autonomous equation}.

\textbf{Step 2}: Verification of the assumptions of Dong and Kim.
The following lemma shows that the mean oscillation condition in Assumption \cite[Asm.~7.1]{Dong-Kim-TAMS} is fulfilled. Hence, \cite[Thm.~7.2 \& Sec.~8]{Dong-Kim-TAMS} is applicable in our setting.

\begin{lemma}\label{lem: integral inequality}
	Let $\gamma \in (0,\nicefrac{1}{4})$.
        Then there exists $R_0 \in (0,1]$ depending only on $\gamma$ and the Hölder regularity of $\{B_t\}_{t \in \R}$ such that, for any $(t,x) \in \R^{d+1}$ and $r \in (0,R_0]$, we have
	\begin{align}
		\barint_{\Q_r(t,x)} \left| B^{\sysalpha\sysbeta}_s(y_1, \hat{y}) - \barint_{\Q_r^\prime(t,\hat{x})} B^{\sysalpha\sysbeta}_\tau(y_1, \hat{z}) \d \hat{z} \d \tau  \right| \d y \d s
		\leq \gamma \quad(\sysalpha, \sysbeta = 1,\dots,m),
	\end{align}
	where $\Q_r$ and $\Q_r^\prime$ denote the \emph{parabolic cylinders} given by
	\begin{align}
		\Q_r(t,x) \coloneqq (t - r^2, t) \times \B_r(x)
		\quad\text{and}\quad
		\Q_r^\prime(t,\hat{x}) \coloneqq (t - r^2, t) \times \B_r^\prime(\hat{x}),
	\end{align}
        respectively, and $x = (x_1, \hat{x})$ with $x_1 \in \R$ and $\hat{x} \in \R^{d - 1}$.
\end{lemma}

\begin{proof}
    Let $r > 0$ and $(t,x) \in \R^{d + 1}$.
    Fix $(s,y) \in \R^{d + 1}$.
	We decompose the integrand as
	\begin{align}
		&\Bigr| B_s^{\sysalpha\sysbeta}(y_1, \hat{y}) - \barint_{\Q_r^\prime(t,\hat{x})} B^{\sysalpha\sysbeta}_\tau(y_1, \hat{z}) \d \hat{z} \d \tau \Bigr| \\
                &\quad\quad\leq
		\barint_{\Q_r^\prime(t,\hat{x})}
		\left| B^{\sysalpha\sysbeta}_s(y) - B^{\sysalpha\sysbeta}_\tau(y) \right|
		+\left| B^{\sysalpha\sysbeta}_\tau(y_1, \hat{y}) - B^{\sysalpha\sysbeta}_\tau(y_1, \hat{z}) \right| \d \hat{z}  \d \tau.
	\end{align}
	Now, for the first term, we have using regularity of $B$ that
        \begin{align}\label{eq: temporal hoelder}
		\left| B^{\sysalpha\sysbeta}_s(y) - B^{\sysalpha\sysbeta}_\tau(y) \right|
                \leq \|B^{\sysalpha\sysbeta}_s - B^{\sysalpha\sysbeta}_\tau \|_\infty
                \lesssim | s - \tau |^\eps
                \lesssim | s - t |^\eps + | t - \tau |^\eps
                \lesssim r^{2 \eps}
	\end{align}
        and, for the second term,
        \begin{align}\label{eq: spatial hoelder}
		\left| B^{\sysalpha\sysbeta}_\tau(y_1, \hat{y}) - B^{\sysalpha\sysbeta}_\tau(y_1, \hat{z}) \right|
                \lesssim | \hat{y} - \hat{z} |^\eps \| B_\tau^{\sysalpha\sysbeta} \|_{\Cont_x^\eps}
                \lesssim (2r)^\eps.
        \end{align}
        Observe that both estimates are uniform in $s$ and $y$, to calculate the average over $\Q_r(t,x)$ as
        \begin{align}
	  \barint_{\Q_r(t,x)} \Bigl| B^{\sysalpha\sysbeta}_s(y_1, \hat{y}) - \barint_{\Q_r^\prime(t,\hat{x})} B^{\sysalpha\sysbeta}_\tau(y_1, \hat{z}) \d \hat{z} \d \tau  \Bigr| \d y \d s
	  \lesssim  r^{2\eps} + (2r)^\eps,
        \end{align}
        where the implicit constant depends on the Hölder regularity of $B$ and $\eps$.
        Now, given $\gamma \in (0,\nicefrac{1}{4})$, choose $R_0 \in (0, 1]$ small enough (depending on the implicit constant) to conclude.
\end{proof}

\begin{remark}
Note that the proof of Lemma~\ref{lem: integral inequality} did not need the full mixed Hölder regularity of $\{B_t\}_{t \in \R}$.
Indeed, the calculations in the proof show that estimates \eqref{eq: temporal hoelder} and \eqref{eq: spatial hoelder} both only rely on Hölder regularity in one of the two variables, uniformly with respect to the other variable.
\end{remark}

\section{Estimates for the solution formula}
\label{Sec: estimates solution formula}

In this section, we consider a family of operators $\{ \cB_t \}_{0<t<T}$ associated with coefficients $B(t,\cdot) \in \E(\Lambda, \lambda, \eps, M)$. Moreover, $B\in \Contt^{\beta+\eps}(\Hx^{\alpha+\eps, \nicefrac{d}{\alpha}})$ if $p < \nicefrac{d}{\alpha}$ and $B\in \Contt^{\beta+\eps}(\Contx^{\alpha+\eps})$ otherwise. The prototype for such a family of operators is the family $\{ \cL_t \}_{0<t<T}$ from Section~\ref{Sec: Introduction}. First, we derive a solution formula for weak $(p,q)$-solutions to the associated non-autonomous problem. Second, we derive suitable estimates for it, which depend heavily on the regularity assumption for the coefficients. Implicit constants are throughout this section allowed to depend on $p$, $q$, $[w]_{\A_q}$, $\Lambda$, $\lambda$, $\alpha$, $\beta$, $\eps$, H\"older regularity, and dimensions.

\subsection{Representation formula by Acquistapace and Terreni}
\label{Subsec: AT formula}

For a weak $(p,q)$-solution $u$ of~\eqref{Eq: shifted non-autonomous equation}, we rely on a well-known representation formula due to Acquistapace and Terreni in $\Wx^{-1,p}$ given pointwise by
\begin{align}
\label{Eq: representation formula}
\tag{$\heartsuit$}
	u(\tt) &= \int_0^\tt \e^{-(\tt-s) (\cB_\tt + \shift)} \bigl( \cB_\tt - \cB_s \bigr) u(s) \d s + \int_0^\tt \e^{-(\tt-s) (\cB_\tt + \shift)} f(s) \d s.
\end{align}
In the unweighted situation, the proof is well-known in the literature~\cite{AT87,Blasio06,HO15,Fackler18}, but we give a streamlined version that directly works with absolute continuity.

\begin{proof}[Proof of~\eqref{Eq: representation formula}]
Consider on $[0,\tt]$ the function $v(s) = \e^{-(\tt-s) (\cB_\tt + \shift)} u(s)$. Moreover, let $0\leq \tau < \tt$. We claim the identity
\begin{align}
\label{Eq: S2 AC}
	v(\tau) = v(0) + \int_0^\tau (\cB_\tt + \shift) \e^{-(\tt-s) (\cB_\tt + \shift)} u(s) + \e^{-(\tt-s) (\cB_\tt + \shift)} u'(s) \d s.
\end{align}
Before we turn to the proof of~\eqref{Eq: S2 AC}, we show how it implies~\eqref{Eq: representation formula}. Note that the function $(\cB_\tt + \shift) \e^{-(\tt-s) (\cB_\tt + \shift)} u(s) + \e^{-(\tt-s) (\cB_\tt + \shift)} u'(s)$ is in $\Lt^q(w;\Wx^{-1,p})$, since $u$ is a weak $(p,q)$-solution (keep Remark~\ref{Rem: weak solution}~\ref{Item: weak derivative} in mind) and the semigroup is bounded on $\Wx^{-1,p}$ owing to Proposition~\ref{Prop: semigroup W-1p bounds}. Hence, by Lebesgue's theorem, we can take the limit $\tau\to \tt$ on the right-hand side of~\eqref{Eq: S2 AC}. Equally, we can take this limit on the left-hand side, owing to the facts that $u$ is uniformly continuous over $[0,\tt]$ with values in $\Wx^{-1,p}$, and the family $\{ \e^{-(\tt-s) (\cB_\tt + \shift)} \}_{0\leq s \leq \tt}$ is strongly continuous and bounded as a family of operators on $\Wx^{-1,p}$. Then, plugging in the actual definition of $v$
and using Remark~\ref{Rem: weak solution}~\ref{Item: weak derivative}
yield~\eqref{Eq: representation formula}.

Let us come back to the proof of~\eqref{Eq: S2 AC}. On the interval $[0,\tau]$, $s\mapsto \e^{-(\tt-s) (\cB_\tt + \shift)}$, considered as a family of operators on $\Wx^{-1,p}$, has a bounded derivative due to Proposition~\ref{Prop: semigroup W-1p bounds} and analyticity. As $u$ is a weak $(p,q)$-solution, $u\colon [0,\tau] \to \Wx^{-1,p}$ is likewise absolutely continuous (see~\cite[Lem.~4.1]{GV17a} for the time-weighted argument). Hence, observing $u(0)=0$, deduce~\eqref{Eq: S2 AC} from Lemma~\ref{Lem: product rule AC} below.
\end{proof}

\begin{lemma}
\label{Lem: product rule AC}
	Let $X$, $Y$ be Banach spaces, $\tau>0$, $\{ T(s) \}_{0\leq s \leq \tau}$ be a differentiable family of operators $X\to Y$ with bounded derivative, and $g \colon [0, \tau] \to X$ be absolutely continuous. Then $s\mapsto T(s)g(s) \in Y$ is an absolutely continuous function on $[0,\tau]$ with derivative $T'(s)g(s) + T(s) g'(s)$.
\end{lemma}

\begin{proof}
	The assumption on $T$ implies in particular that $s\mapsto T(s)$ is absolutely continuous on $[0,\tau]$. Now, use absolute continuity of both $T$ and $g$, and the Fubini--Tonelli theorem, to give
	\begin{align}
		&\int_0^\tau T'(s) g(s) + T(s) g'(s) \d s \\
		={} &\int_0^\tau T'(s) \Bigl( g(0) + \int_0^s g'(u) \d u \Bigr) \d s + \int_0^\tau \Bigl( T(0) + \int_0^s T'(u) \d u \Bigr) g'(s) \d s \\
		={} &\int_0^\tau T'(s) \d s \, g(0) + T(0) \int_0^\tau g'(s) \d s + \int_0^\tau \int_0^\tau T'(s) g'(u) \d u \d s.
	\end{align}
	All remaining integrals can now be evaluated using absolute continuity, and we only remain with $T(\tau)g(\tau) - T(0)g(0)$ after having canceled all superfluous terms. Rearranging terms gives the claim.
\end{proof}

Motivated by~\eqref{Eq: representation formula}, we are going to consider the operators
\begin{align}
\label{Eq: Def S1 and S2}
\begin{split}
	S_1(u)(\tt) \mapsto &\int_0^\tt (\cB_\tt + \shift) \e^{-(\tt-s) (\cB_\tt + \shift)} (\cB_\tt - \cB_s) u(s) \d s, \\
	S_2(f)(\tt) \mapsto &(\cB_\tt + \shift) \int_0^\tt \e^{-(\tt-s) (\cB_\tt + \shift)} f(s) \d s.
\end{split}
\end{align}
Up to some technicalities, boundedness of $S_1$ and $S_2$ will lead to the maximal regularity estimate for $u$ later on in Section~\ref{Sec: proof main result}.

\subsection{Estimates for the kernel of \texorpdfstring{\boldsymbol{$S_1$}}{S1}}

The following lemma is simple, but central in our argument, as it is the only result that uses the full \emph{simultaneous regularity} in the spatial and temporal variables.

\begin{lemma}
\label{Lem: estimate kernel of S1 part 1}
	Let $s\in (0, \tt)$.
	The operator $\cB_\tt - \cB_s$ acts as a bounded operator $\Wx^{1+\alpha,p} \to \Wx^{-1+\alpha,p}$ along with the estimate
	\begin{align}
		\| \cB_\tt - \cB_s \|_{\Wx^{1+\alpha,p} \to \Wx^{-1+\alpha,p}} \lesssim |\tt-s|^{\beta+\eps}.
	\end{align}
\end{lemma}

\begin{proof}
	Let $s\in (0, \tt)$ and $f\in \Wx^{1+\alpha, p} \cap \Wx^{1,2}$.
	Put $X = \Hx^{\alpha+\eps,\nicefrac{d}{\alpha}}$ if $p < \nicefrac{d}{\alpha}$ and $X = \Contx^{\alpha+\eps}$ otherwise,
	and recall from Lemma~\ref{Lem: multiplier} that an $X$-function is a multiplier on the space $\Wx^{\alpha,p}$, and that its operator norm can be controlled by its $X$-norm. Hence, for $g\in \Wx^{1-\alpha, p'} \cap \Wx^{1,2}$, estimate
	\begin{align}
		|\langle (\cB_\tt - \cB_s) f, g \rangle| &= \bigl|\int_{\R^d} (B(\tt, x) - B(s, x)) \nabla f(x) \cdot \overline{\nabla g(x)} \d x \bigr| \\
		&\leq \| (B(\tt, \cdot) - B(s, \cdot)) \nabla f \|_{\Wx^{\alpha,p}} \| \nabla g \|_{\Wx^{-\alpha, p'}} \\
		&\lesssim \| B(\tt, \cdot) - B(s, \cdot) \|_X \| \nabla f \|_{\Wx^{\alpha,p}} \| g \|_{\Wx^{1-\alpha, p'}}.
	\end{align}
	Using the regularity of $A$ and duality, we deduce
	\begin{align}
		\| (\cB_\tt - \cB_s) f \|_{\Wx^{-1+\alpha,p}} \lesssim |\tt-s|^{\beta+\eps} \| \nabla f \|_{\Wx^{\alpha,p}} \leq |\tt-s|^{\beta+\eps} \| f \|_{\Wx^{1+\alpha,p}}. &\qedhere
	\end{align}
\end{proof}

\begin{lemma}
\label{Lem: estimate kernel of S1 part 2}
	Let $s \in (0,\tt)$. The operator $(\cB_\tt + \shift) \e^{-(\tt-s) (\cB_\tt + \shift)}$ acts as a bounded operator $\Wx^{-1+\alpha,p} \to \Lx^p$, and satisfies the estimate
	\begin{align}
		\| (\cB_\tt + \shift) \e^{-(\tt-s) (\cB_\tt + \shift)} \|_{\Wx^{-1+\alpha,p} \to \Lx^p} \lesssim |\tt-s|^{-\beta-1}.
	\end{align}
\end{lemma}

\begin{proof}
	By duality, it suffices to show that $((B_\tt)^* + \shift) \e^{-(\tt-s) ((B_\tt)^* + \shift)}$ maps $\Lx^{p'} \to \Wx^{1-\alpha,p'}$ with norm controlled by $|\tt-s|^{-\beta-1}$. We are going to show that $((B_\tt)^* + \shift) \e^{-(\tt-s) ((B_\tt)^* + \shift)}$ maps $\Lx^{p'} \to \Lx^{p'}$ with estimate against $|\tt-s|^{-1}$ and $\Lx^{p'} \to \Wx^{1,p'}$ with norm controlled by $|\tt-s|^{-\nicefrac{3}{2}}$. Then the claim is a consequence of complex interpolation, keeping the relation $\nicefrac{(\alpha-3)}{2} = -\beta-1$ in mind.

	Note that as a consequence of Theorem~\ref{Thm: Lp bounds semigroup}, the $\mathrm{H}^\infty$-calculus of $(B_\tt)^*$ is bounded on $\Lx^{p'}$ with control of the implicit constants. Hence, in conjunction with Theorem~\ref{Thm: square root Lp bounds}, calculate
	\begin{align}
		\| ((B_\tt)^* + \shift) \e^{-(\tt-s) ((B_\tt)^* + \shift)} f \|_{\Wx^{1,p'}} &\lesssim \| ((B_\tt)^* + \shift)^\frac{3}{2} \e^{-(\tt-s) ((B_\tt)^* + \shift)} f \|_{\Lx^{p'}} \\
		&= (\tt-s)^{-\nicefrac{3}{2}} \| [\z^{\nicefrac{3}{2}} \e^{-\z}]((\tt-s)((B_\tt)^*+\shift)) f \|_{\Lx^{p'}} \\
		&\lesssim (\tt-s)^{-\nicefrac{3}{2}} \| f\|_{\Lx^{p'}}.
	\end{align}
	The calculation for $\Lx^{p'} \to \Lx^{p'}$ is similar. This completes the proof.
\end{proof}

\subsection{Boundedness of \texorpdfstring{\boldsymbol{$S_2$}}{S2}}
\label{Subsec: S2 bdd}

Recall the operator $S_2$ from~\eqref{Eq: Def S1 and S2}.
The aim of this subsection is to show the following.
\begin{proposition}
\label{Prop: S2 bdd}
	Let $p,q\in (1,\infty)$, $w\in \A_q$, then one has the estimate
	\begin{align}
	\label{Eq: S2 boundedness estimate}
		\| S_2 f \|_{\Lt^q(w;\Lx^p)} \lesssim \| f \|_{\Lt^q(w;\Lx^p)} \qquad (f\in \Cont_0^\infty(\Lx^p\cap \Lx^2)).
	\end{align}
	Implicit constants only depend on the quantities from Agreement~\ref{Agreement: fix constants}.
\end{proposition}

In the unweighted case, it is well-known in the literature~\cite{HM00,PS06,Fackler18} that such bounds for $S_2$ follow from the boundedness of some pseudo-differential operator with operator-valued kernel. For the reader's convenience, we include a proof. For further background on such pseudo-differential operators, the reader may consult~\cite{HP,PS06} and the references therein. In Remark~\ref{Rem: pseudo weighted}, we will comment on the extension to the weighted setting.

For technical reasons, we extend $f$ by $0$ outside of $(0,T)$, and we extend the operator family $\{ \cB_t \}_{0<t<T}$ to $\R$ constantly at the endpoints (we performed the same extension already in Section~\ref{Sec: existence weak solutions}). Using the vector-valued Fourier transform $\FT$ (see~\cite[Sec.~2.4.c]{AiBSVolI} for further information) and the Fubini--Tonelli theorem (its application is justified by integrability of $\FT f$ and exponential decay of the semigroup), calculate
\begin{align}
\begin{split}
\label{Eq: S2 identity I}
	\int_0^\tt \e^{-(\tt-s) (\cB_\tt + \shift)} f(s) \d s &= \int_{-\infty}^\infty \e^{-(\tt-s) (\cB_\tt + \shift)} \ind_{[0,\infty)}(\tt-s) f(s) \d s \\
	&= \int_{-\infty}^\infty \e^{-(\tt-s) (\cB_\tt + \shift)} \ind_{[0,\infty)}(\tt-s) \int_{-\infty}^\infty \FT f(\tau) \e^{2\pi i s \tau} \d \tau \d s \\
	&= \int_{-\infty}^\infty \int_{-\infty}^\infty \e^{-(\tt-s) (\cB_\tt + \shift)} \ind_{[0,\infty)}(\tt-s) \e^{2\pi i s \tau} \d s \, \FT f(\tau) \d \tau \\
	&= \int_{-\infty}^\infty \I(\tau, \tt) \FT f(\tau) \d \tau,
\end{split}
\end{align}
where $\I(\tau, \tt)$ is implicitly defined by the latest identity. Using the transformation $u=\tt-s$ and the relationship between a semigroup and its generator in virtue of the Laplace transform (apply for instance~\cite[Prop.~3.4.1 d)]{Haase} to $\cB_\tt + \nicefrac{\shift}{2}$), deduce
\begin{align}
	\I(\tau, \tt) = \e^{2\pi i \tau \tt} \int_0^\infty \e^{-u (\cB_\tt + \shift)} \e^{-2\pi i u \tau} \d u = \e^{2\pi i \tau \tt} (2\pi i \tau + (\cB_\tt + \shift))^{-1}.
\end{align}
Plug this back into~\eqref{Eq: S2 identity I} to conclude with the definition of $S_2$ that
\begin{align}
\label{Eq: S2 identity II}
	S_2(f)(\tt) = (\cB_\tt + \shift)  \int_{-\infty}^\infty (2\pi i \tau + (\cB_\tt + \shift))^{-1} \FT f(\tau) \e^{2\pi i \tau \tt} \d \tau.
\end{align}
The integral $\int_{-\infty}^\infty \| (\cB_\tt + \shift) (2\pi i \tau + (\cB_\tt + \shift))^{-1} \FT f(\tau) \e^{2\pi i \tau \tt } \|_{\Lx^2} \d \tau$ is finite, so we can commute $(\cB_\tt + \shift)$ with the integral in~\eqref{Eq: S2 identity II} owing to Hille's theorem.
This means that $S_2(f)$ can be represented as the pseudo-differential operator with symbol $$(\tau,s) \mapsto (\cB_s + \shift) (2\pi i \tau + (\cB_s + \shift))^{-1}.$$ Of course, by expansion, we can equally study boundedness of the pseudo-differential operator associated with the symbol $(\tau, s) \mapsto 2\pi i \tau (2\pi i \tau + (\cB_s + \shift))^{-1}$. In the following lemma, we study this symbol thoroughly.
\begin{lemma}
\label{Lem: R-Yamazaki}
	For all $\ell \geq 0$, the symbol $a(\tau, s) = 2\pi i \tau (2\pi i \tau + (\cB_s + \shift))^{-1}$ is in $\mathcal{S}^0_{1,0}(\eps, \ell, \Lx^p)$, that is to say, there is some constant $C>0$ such that, for $k=0,\dots,\ell$, one has the $R$-bound
	\begin{align}
		\Rbound \bigl\{ (1+|\tau|)^k \partial_\tau^k a(\tau, s) \SetSep s,\tau \in \R \bigr\} \leq C,
	\end{align}
	and for $s,h,\tau \in \R$ one has the regularity condition
	\begin{align}
		\bigl\| \partial_\tau^k \bigl[ a(\tau, s) - a(\tau, s+h) \bigr] \bigr\|_{\Lx^p \to \Lx^p} \leq C |h|^\eps (1+|\tau|)^{-k}.
	\end{align}
\end{lemma}
\begin{proof}
	For brevity, we rescale $2 \pi \tau$ to $\tau$ in the definition of the symbol $a$.
	Fix $\varphi \in (\omega, \nicefrac{\pi}{2})$ and $\psi \in (\nicefrac{\pi}{2}, \pi)$ such that $\varphi + \psi < \pi$, and let $s,h\in \R$. Define on $\Sec_\psi \cup \B(0,\nicefrac{\shift}{2})$ the function $A(\lambda) = (1+\lambda) \bigl[ (\lambda + (\cB_s + \shift))^{-1} - (\lambda + (\cB_{s+h} + \shift))^{-1} \bigr]$.

	\textbf{Step 1}: Reduction to the case $k=0$. Since the function $A$ is holomorphic in $\lambda$ and defined on a sector that strictly includes the half-plane as well as a ball around the origin of fixed radius, the regularity condition reduces, as a consequence of Cauchy's formula for derivatives, to boundedness of $A$, which is the case $k=0$ (but in a larger region). With a similar auxiliary function, the same is true for the $R$-boundedness condition, see~\cite[Ex.~2.16]{KW}.

	\textbf{Step 2}: Verification of the case $k=0$. The $R$-boundedness condition follows directly from Corollary~\ref{Cor: R-bounded resolvent}. Hence, it only remains to show that the function $A$ is bounded in operator norm with control against $|h|^\eps$. For $\lambda\in \Sec_\psi \cup \B(0,\nicefrac{\shift}{2})$, expand $A(\lambda)$ using the functional calculus as
		\begin{align}
		 (1+\lambda)^\frac{1}{2} \left[ \frac{\z^\frac{1}{2}}{\lambda + \z + \nicefrac{\shift}{2}} \right](B_s + \nicefrac{\shift}{2}) \tag{F1}\label{eq:F1}\\
		 \times{} (\cB_s + \nicefrac{\shift}{2})^{-\frac{1}{2}} \bigl[ \cB_{s+h} - \cB_s \bigr] (\cB_{s+h} + \nicefrac{\shift}{2})^{-\frac{1}{2}} \tag{F2}\label{eq:F2}\\
		\times{} (1+\lambda)^\frac{1}{2} \left[ \frac{\z^\frac{1}{2}}{\lambda + \z + \nicefrac{\shift}{2}} \right](B_{s+h} + \nicefrac{\shift}{2}). \tag{F3}\label{eq:F3}
		\end{align}
	Using composition, we can treat all three factors separately. The decay in $|h|$ comes from \eqref{eq:F2}, whereas the other two are merely bounded. Moreover, \eqref{eq:F1} and \eqref{eq:F3} have the same structure, so we only present the estimate for \eqref{eq:F1}. Recall that, according to Remark~\ref{Rem: DK works with half shift}, all results from Section~\ref{Sec: Elliptic} can be applied to $\cB_s + \nicefrac{\shift}{2}$.

	Define on $\Sec_\varphi$ the function $g_\lambda=\z^\frac{1}{2} (\lambda + \z + \nicefrac{\shift}{2})^{-1}$. As a consequence of Theorem~\ref{Thm: Lp bounds semigroup}, the $\mathrm{H}^\infty$-calculus of $B_s + \nicefrac{\shift}{2}$ is bounded on $\Lx^p$, compare with the proof of Lemma~\ref{Lem: estimate kernel of S1 part 2}. This means that we have to bound $\| g_\lambda \|_\infty$ in an appropriate way. Using the reverse triangle inequality on sectors and the case distinction $|z| \geq |\lambda| + \nicefrac{\shift}{2}$ and $|z| \leq |\lambda| + \nicefrac{\shift}{2}$, we indeed find readily $\| g_\lambda \|_\infty \lesssim (\nicefrac{\shift}{2}+|\lambda|)^{-\nicefrac{1}{2}}$. Using $(\nicefrac{\shift}{2}+|\lambda|)^{-\nicefrac{1}{2}} \approx (1+|\lambda|)^{-\nicefrac{1}{2}}$, the factor in front of $g_\lambda(\cB_s + \nicefrac{\shift}{2})$ cancels out, which completes the treatment of~\eqref{eq:F1}.

	It remains to treat \eqref{eq:F2}.
	The crucial ingredient is the estimate
	\begin{align}
	\label{Eq: Hoelder for elliptic operators}
		\| \cB_{s} - \cB_{s+h} \|_{\Wx^{1,p} \to \Wx^{-1,p}} \lesssim |h|^\eps,
	\end{align}
	whose proof follows the lines of Lemma~\ref{Lem: estimate kernel of S1 part 1}, but it suffices to have coefficients in $\Cont^\eps(\R; \Lx^\infty)$.
	Recall from Theorem~\ref{Thm: square root Lp bounds} the estimate $\| (\cB_s + \nicefrac{\shift}{2})^{-\frac{1}{2}} f \|_{\Wx^{1,p}} \lesssim \| f \|_{\Lx^p}$ for $f\in \Lx^p \cap \Lx^2$. The same estimate holds of course if $\cB_s$ and $p$ are replaced by $(\cB_s)^*$ and $p'$. Hence, we can estimate by duality and using~\eqref{Eq: Hoelder for elliptic operators} that, for $g \in \Lx^{p'}$,
	\begin{align}
		&|( (\cB_s + \nicefrac{\shift}{2})^{-\frac{1}{2}} (\cB_{s+h} - \cB_s) (\cB_{s+h} + \nicefrac{\shift}{2})^{-\frac{1}{2}} f \SP g)| \\
		={} &|( (\cB_{s+h} - \cB_s) (\cB_{s+h} + \nicefrac{\shift}{2})^{-\frac{1}{2}} f \SP ((\cB_s)^* + \nicefrac{\shift}{2})^{-\frac{1}{2}} g)| \\
		\leq{} &\| (\cB_{s+h} - \cB_s) (\cB_{s+h} + \nicefrac{\shift}{2})^{-\frac{1}{2}} f \|_{\Wx^{-1,p}} \| ((\cB_s)^* + \nicefrac{\shift}{2})^{-\frac{1}{2}} g \|_{\Wx^{1,p'}} \\
		\lesssim{} &|h|^\eps \| (\cB_{s+h} + \nicefrac{\shift}{2})^{-\frac{1}{2}} f \|_{\Wx^{1,p}} \| g \|_{\Lx^{p'}} \\
		\lesssim{} &|h|^\eps \| f \|_{\Lx^{p}} \| g \|_{\Lx^{p'}}.
	\end{align}
	Consequently,
	\begin{align}
		\| (\cB_s + \nicefrac{\shift}{2})^{-\frac{1}{2}} (\cB_{s+h} - \cB_s) (\cB_{s+h} + \nicefrac{\shift}{2})^{-\frac{1}{2}} \|_{\Lx^p \to \Lx^p}
		\lesssim |h|^\eps. &\qedhere
	\end{align}
\end{proof}

\begin{proof}[Proof of Proposition~\ref{Prop: S2 bdd}]
	We have already seen that the bound for $S_2$ follows from the bound for the pseudo-differential operator associated with the symbol $a(\tau, t) = 2\pi i \tau (2\pi i \tau + (\cB_t + \shift))^{-1}$. It was shown in~\cite[Thm.~5]{PS06} that boundedness for such a pseudo-differential operator follows if the symbol $a$ is in $\mathcal{S}^0_{1,0}(\eps, 6, \Lx^p)$ and $w=1$. The condition on $a$ was just verified in Lemma~\ref{Lem: R-Yamazaki}. Dependence of implicit constants becomes apparent from an inspection of the proof. Moreover, the result in~\cite{PS06} extends to $w\in \A_q$, see Remark~\ref{Rem: pseudo weighted}.
\end{proof}

\begin{remark}[Weighted operator-valued pseudo-differential operators]
\label{Rem: pseudo weighted}
	The proof of~\cite[Thm.~5]{PS06} consists of 4 steps: 1) Decomposition of a general symbol into an \enquote{error symbol} and a symbol that is smooth in $s$. 2) Representation of smooth symbols by elementary symbols. 3) Estimate for pseudo-differential operators associated with an error symbol. 4) Estimate for pseudo-differential operators associated with an elementary symbol. Steps 1) and 2) stay, of course, valid. In Step 3), Schur's test is used, but the kernel estimate directly falls under the scope of Lemma~\ref{Lem: weighted convolution}. Finally, in Step 4), a vector-valued Littlewood--Paley decomposition, a vector-valued Mikhlin theorem and $R$-boundedness of Littlewood--Paley operators are used. These ingredients remain true in the weighted setting, see~\cite{MV} or~\cite{Nick}. Finally, Schur's test is used once again, this time with a more complicated kernel bound, which nevertheless can be captured by Lemma~\ref{Lem: weighted convolution}.
\end{remark}

\section{Higher regularity of weak solutions}
\label{Sec: higher regularity}

In this section, we consider a family of operators $\{ \cB_t \}_{0<t<T}$ associated with coefficients $B(t,\cdot) \in \E(\Lambda, \lambda, \eps, M)$. If $p < \nicefrac{d}{\alpha}$, we assume in addition that $B\in \Lt^\infty(\Hx^{\alpha+\eps, \nicefrac{d}{\alpha}})$, otherwise we require $B\in \Lt^\infty(\Contx^{\alpha+\eps})$. Note that we do not require any regularity in time in this section. Provided that the associated problem~\eqref{Eq: shifted non-autonomous equation} admits a solution, we show higher spatial regularity for this solution in Proposition~\ref{Prop: higher regularity}. This is based on a commutator argument that already appeared in~\cite{AE16}. Implicit constants are allowed to depend on $p$, $q$, $[w]_{\A_q}$, $\Lambda$, $\lambda$, $\alpha$, $\eps$, $\shift$, H\"older constants, and dimensions.

Recall the fractional derivative $\partial^\alpha_x$ from Section~\ref{Sec: spaces and weights}. We use the representation of $\partial^\alpha_x$ as a hypersingular integral to show the following commutator estimates.

\begin{lemma}[Commutator estimates]
\label{Lem: commutator estimate}
	Let $p\in (1,\infty)$. Assume that $b$ is a smooth and bounded scalar function on $\R^d$.
	Then the commutator $[\partial_x^\alpha, b] \coloneqq \partial_x^\alpha b - b \partial_x^\alpha$, initially defined on $\Wx^{\alpha,p}$, extends to a bounded operator on $\Lx^p$,
	and satisfies the estimate
	\begin{align}
	\label{Eq: commutator estimate}
		\| [\partial_x^\alpha, b] f \|_{\Lx^p} \lesssim \| b \|_{\Contx^{\alpha+\eps}} \| f \|_{\Lx^p} \qquad (f\in \Wx^{\alpha,p}).
	\end{align}
	Moreover, if $p < \nicefrac{d}{\alpha}$, then, for all $\eps > 0$, there exists a constant $C_\eps > 0$ such that
	\begin{align}
		\| [\partial_x^\alpha, b] f \|_{\Lx^p} \leq C_\eps \| b \|_{\Lx^\infty} \| f\|_{\Lx^p} + \eps \| b \|_{\Hx^{\alpha+\eps, \nicefrac{d}{\alpha}}} \| \partial_x^\alpha f \|_{\Lx^p} \qquad (f\in \Wx^{\alpha,p}).
	\end{align}
\end{lemma}

\begin{proof}
	Observe that, since $[\partial_x^\alpha, b] \colon \Wx^{\alpha,p} \to \Lx^p$ is bounded, it suffices, in virtue of density and Fatou's lemma, to establish~\eqref{Eq: commutator estimate} for $f$ smooth and bounded.

	According to~\cite[Sec.~25.4]{Samko}, the fractional derivative $\partial_x^\alpha$ acts on bounded and smooth functions $g$ as the \emph{hypersingular integral} given for $x\in \R^d$ by
	\begin{align}
		\partial_x^\alpha g(x) = c \int_{\R^d} \frac{g(y)-g(x)}{|y-x|^{d+\alpha}} \d y.
	\end{align}
	We can apply this identity to $f$ and $b f$ in virtue of the assumption on $b$ and the reduction at the beginning of this proof. Consequently, the commutator can be written as
	\begin{align}
		[\partial_x^\alpha, b] f(x) &= c \Bigl(\int_{\R^d} \frac{b(y)f(y)-b(x)f(x)}{|y-x|^{d+\alpha}} \d y - b(x) \int_{\R^d} \frac{f(y)-f(x)}{|y-x|^{d+\alpha}} \d y \Bigr) \\
		&= c \int_{\R^d} \frac{(b(y)-b(x))f(y)}{|y-x|^{d+\alpha}} \d y.
	\end{align}
	Let $h\leq 1$ and split the integral into the regions $|y-x| \leq h$ and $|y-x| \geq h$ to rewrite the latest expression as
	\begin{align}
		c\int_{|y-x| \geq h} \frac{(b(y)-b(x))f(y)}{|y-x|^{d+\alpha}} \d y + c\int_{|y-x| \leq h} \frac{(b(y)-b(x))f(y)}{|y-x|^{d+\alpha}} \d y
		\eqqcolon \I + \II.
	\end{align}
	Use boundedness of $b$ for term $\I$ to estimate $|\I| \lesssim \| b\|_{\Lx^\infty} \bigl(\ind_{|\cdot| \geq h} |\cdot|^{-d-\alpha} \ast |f|\bigr)(x)$. Note that $\| \ind_{|\cdot| \geq h} |\cdot|^{-d-\alpha} \|_1 = h^{-\alpha} \| \ind_{|\cdot| \geq 1} |\cdot|^{-d-\alpha} \|_1 \lesssim h^{-\alpha}$ by scaling.

	\textbf{Part 1}: H\"older coefficients.
	We specify $h=1$. Use H\"older-regularity of $b$ to bound $|\II| \lesssim [ b ]_{\Contx^{\alpha+\eps}} \bigl( \ind_{|\cdot| \leq 1} |\cdot|^{-d+\eps} \ast |f|\bigr)(x)$. The convolution kernel $\ind_{|y| \leq 1} |y|^{-d+\eps}$ is likewise integrable. In summary, Young's convolution inequality yields the claim.

	\textbf{Part 2}: Sobolev coefficients.
	Write $|y-x|^\alpha = |y-x|^{-\nicefrac{\eps}{4}} |y-x|^{\alpha+\nicefrac{\eps}{4}}$, and use H\"older's inequality to estimate
	\begin{align}
		|\II| \leq \Bigl( \int_{|y-x| \leq 1} |y-x|^{-d+\nicefrac{\eps p'}{4}} \d y \Bigr)^\frac{1}{p'} \Bigl( \int_{|y-x| \leq h} \left| \frac{|b(y)-b(x)||f(y)|}{|y-x|^{\alpha+\nicefrac{\eps}{4}}} \right|^p \frac{\d y}{|y-x|^d} \Bigr)^\frac{1}{p}.
	\end{align}
	The first factor is bounded by a constant depending on $d$, $p$, and $\eps$. Using this estimate and the bound for $\I$ in conjunction with Young's convolution inequality yields
	\begin{align}
	\label{Eq: Sobolev commutator estimate}
		\| [\partial_x^\alpha, b] f \|_{\Lx^p} \lesssim h^{-\alpha} \| b \|_{\Lx^\infty} \| f \|_{\Lx^p} + \Bigl( \int_{\R^d} \int_{|y-x| \leq h} \left| \frac{(b(y)-b(x))f(y)}{|y-x|^{\alpha+\nicefrac{\eps}{4}}} \right|^p \frac{\d y \d x}{|y-x|^d} \Bigr)^\frac{1}{p}.
	\end{align}
	Now use H\"older's inequality with $\nicefrac{1}{p} - \nicefrac{\alpha}{d} \eqqcolon \nicefrac{1}{q}$ (observe that $q$ is finite by the assumption $p < \nicefrac{d}{\alpha}$) to bound the second term in~\eqref{Eq: Sobolev commutator estimate} by
	\begin{align}
		\Bigl( \int_{\R^d} \int_{|y-x| \leq h} \left| \frac{b(y)-b(x)}{|y-x|^{\alpha+\nicefrac{\eps}{2}}} \right|^\frac{d}{\alpha} \frac{\d y \d x}{|y-x|^d} \Bigr)^\frac{\alpha}{d} \times \Bigl( \int_{\R^d} \int_{|y-x| \leq h} |f(y)|^q |y-x|^{\nicefrac{q \eps}{4}} \frac{\d y \d x}{|y-x|^d} \Bigr)^\frac{1}{q}.
	\end{align}
	The first factor is dominated by the $\B^{\alpha+\nicefrac{\eps}{2}}_{\nicefrac{d}{\alpha},\nicefrac{d}{\alpha}}$ norm of $b$, which in turn is under control by the $\Hx^{\alpha+\eps, \nicefrac{d}{\alpha}}$ norm of $b$. Furthermore, the second factor is controlled by $h^{\nicefrac{\eps}{4}} \| f \|_{\Lx^q}$ in virtue of Fubini's theorem. Finally, we bound $\| f \|_{\Lx^q} \lesssim \| \partial_x^\alpha f \|_{\Lx^p}$ using boundedness of the fractional integral, see for instance~\cite[Thm.~2.5.2]{Samko}, where we use again the restriction on $p$.

	Plugging everything back into~\eqref{Eq: Sobolev commutator estimate} gives
	\begin{align}
		\| [\partial_x^\alpha, b] f \|_{\Lx^p} \lesssim h^{-\alpha} \| b \|_{\Lx^\infty} \| f \|_{\Lx^p} + h^{\nicefrac{\eps}{4}} \| b \|_{\Hx^{\alpha+\eps, \nicefrac{d}{\alpha}}} \| \partial_x^\alpha f \|_{\Lx^p}.
	\end{align}
	Choosing $h$ sufficiently small gives the claim.
\end{proof}

\begin{proposition}
\label{Prop: higher regularity}
	Given a weak $(p,q)$-solution $u$ of~\eqref{Eq: shifted non-autonomous equation} for some right-hand side $f\in \Lt^q(w;\Lx^p)$, one has higher spatial regularity in the sense $u\in \Lt^q(w;\Wx^{1+\alpha,p})$ together with the estimate
	\begin{align}
	\label{Eq: higher regularity estimate}
		\| u \|_{\Lt^q(w;\Wx^{1+\alpha,p})} \lesssim \| f\|_{\Lt^q(w;\Lx^p)}.
	\end{align}
\end{proposition}

\begin{proof}
	The proof divides into four steps.

	\textbf{Step 1}: Regularization of the equation.
	Let $\rho \in \Cont^\infty_0(\R^d)$ be positive with integral one and define the usual mollifier sequence $\rho_n(x) \coloneqq n^d \rho(n x)$. Put $B_n \coloneqq \rho_n \ast_x B$, where $\ast_x$ denotes convolution in the $x$-variable. One has
	\begin{align}
		B_n(t,x) \xi \cdot \eta = \int_{\R^d} \rho_n(y) B(t,x-y) \xi \cdot \eta \d y \qquad (\xi, \eta \in \C^{dm}),
	\end{align}
	hence $B_n$ is elliptic with the same bounds as $B$. In conjunction with the calculation
	\begin{align}
	\label{Eq: regularized coefficients Hoelder}
		\| B_n(t,x) - B_n(t,y) \| &\leq \int_{\R^d} \rho_n(z) \| B(t,x-z) - B(t,y-z) \| \d z \\
		&\leq M |x-y|^{\eps},
	\end{align}
	this shows that $B_n$ is again in the class $\E(\Lambda, \lambda, \eps, M)$. If $p \geq \nicefrac{d}{\alpha}$, the calculation in~\eqref{Eq: regularized coefficients Hoelder} moreover shows $B_n \in \Lt^\infty(\Contx^{\alpha+\eps})$, where the norm is controlled by $M$. Otherwise, $B_n \in \Lt^\infty(\Hx^{\alpha+\eps, \nicefrac{d}{\alpha}})$, since the Bessel potential commutes with mollification. Similarly to~\eqref{Eq: regularized coefficients Hoelder}, we derive for fixed $n$ using smoothness of $\rho$ that $B_n$ is Lipschitz in the $x$ variable uniformly in $t$.
	Now, according to Theorem~\ref{Thm: weak solutions}, there exist unique weak $(p,q)$-solutions $u_n$ to equation~\eqref{Eq: shifted non-autonomous equation} with $B$ replaced by $B_n$ in the definition of $\cB_t$.

	\textbf{Step 2}: Qualitative higher regularity for solutions of the regularized equations.
	Using the method of difference quotients, we show that the solutions $u_n$ from Step~1 belong to the class $\Lt^q(w;\Wx^{2,p})$. This is a non-quantitative technical necessity to justify certain calculations in Step~3. To keep the notation concise, we will omit the subscript $n$ and simply write $u$ instead of $u_n$ for the solution, and $B$ instead of $B_n$ for the coefficients. We emphasize that, in this step, the only quantitative property of the regularized coefficients that we are going to use is the Lipschitz property in $x$ uniform in $t$.

	For $y\in \R^d$, define the translation operator $S_y$ in the $x$-variable by $f\mapsto f(\cdot + y)$. We extend $S_y$ by pointwise application in $t$ to parabolic spaces like $\Lt^q(w;\Lx^p)$ (for simplicity, we keep writing the symbol $S_y$ for this extension). Then, set for $j=1,\dots,d$ and $h\in \R$ the difference quotient operator $D^j_h u \coloneqq \frac{1}{h}(S_{h \mathrm{e}_j} u - u)$, where $\mathrm{e}_j$ is the $j$th unit vector in $\R^d$. Observe that the operator $D^j_h$ leaves the space of test functions invariant.

	Using the chain rule and translation in the $x$-variable, one gets for $\tt$ fixed, $y\in \R^d$, and $g\in \Wx^{1,p'}$ the identity
	\begin{align}
	\label{Eq: form applied to translation}
	\begin{split}
		&b_\tt(S_{y} u(\tt), g) \\
		={} &\int_{\R^d} B(\tt, x) \nabla u(\tt, x+y) \cdot \overline{\nabla g(x)} \d x \\
		={} &\int_{\R^d} \bigl[ B(\tt,x) - B(\tt, x+y) \bigr] \nabla u(\tt, x+y) \cdot \overline{\nabla g(x)} \d x + b_\tt(u(\tt), S_{-y} g).
	\end{split}
	\end{align}
	Note that $S_{-y}$ is the adjoint of $S_{y}$, and, consequently, $-D^j_{-h}$ is the adjoint of $D^j_h$. Hence, if we plug $D^j_h u$ in \eqref{Eq: weak pq-solution}, and use the adjoint of $D^j_h$ for the first and third, and~\eqref{Eq: form applied to translation} for the second term, we obtain
	\begin{align}
		&\int_0^T -\varphi'(s) (D^j_h u(s) \SP g) + \varphi(s) b_s(D^j_h u(s), g) + \varphi(s) \shift ( D^j_h u(s) \SP g) \d s \\
		={} &\int_0^T \varphi(s) \int_{\R^d} \Bigl( \frac{B(s,x)-B(s,x+h \mathrm{e}_j)}{h} \Bigr) \nabla u(s,x+h \mathrm{e}_j) \cdot \overline{\nabla g(x)} \d x \d s \\
		&-\int_0^T -\varphi'(s) (u(s) \SP D^j_{-h} g ) + \varphi(s) b_s(u(s), D^j_{-h} g ) + \varphi(s) \shift (u(s) \SP D^j_{-h} g ) \d s \\
		\eqqcolon{} &\mathrm{I} + \mathrm{II}.
	\end{align}
	To bound term $\mathrm{II}$, we use first that $u$ is a solution for the right-hand side $f$, followed by the fact that we can estimate the difference quotients of $g$ by $\nabla g$, see for instance~\cite[Sec.~5.8.2.~Thm.~3]{Evans}. So, write
	\begin{gather}
		\int_0^T -\varphi'(s) (u(s) \SP D^j_{-h} g ) + \varphi(s) b_s\bigl(u(s), D^j_{-h} g \bigr) + \varphi(s) \shift (u(s) \SP D^j_{-h} g ) \d s \\
                = \int_0^T \varphi(s) ( f(s) \SP D^j_{-h} g) \d s,
	\end{gather}
	and for $s\in (0,T)$ fixed and all $h\in \R$, estimate the pairing in its integrand by
	\begin{align}
		|(f(s) \SP D^j_{-h} g )| \leq \| f(s) \|_{\Lx^p} \| D^j_{-h} g \|_{\Lx^{p'}} \lesssim \| f(s) \|_{\Lx^p} \| \nabla g \|_{\Lx^{p'}}.
	\end{align}
	Using H\"{o}lder's inequality in the $t$-variable reveals that term $\mathrm{II}$ defines a functional on $\Lt^{q'}(w';\Wx^{1,p'})$ and is thus induced by a function in $\Lt^q(w;\Wx^{-1,p})$, with bound independent of $h$. To treat term~$\mathrm{I}$, use that $B$ is Lipschitz in the $x$-variable uniformly in $s\in (0,T)$, along with H\"{o}lder's inequality and translation invariance of the $\Lx^p$-norm.

	Eventually, we see that $D^j_h u$ is a weak $(p,q)$-solution to some right-hand side in $\Lt^q(w;\Wx^{-1,p})$, where the norm of the right-hand side can be controlled independently of $h$. Consequently, the estimate from Theorem~\ref{Thm: weak solutions} gives
	\begin{align}
	\label{Eq: difference quotients uniformly bounded I}
		\| D^j_h u \|_{\Lt^q(w;\Wx^{1,p})} \lesssim_n \| f \|_{\Lt^q(w;\Lx^p)} \qquad (h\in \R, j=1,\dots,d).
	\end{align}
	With the symbol $\lesssim_n$ we emphasize that the implicit constant here depends on the regularization from Step~1.
	In particular, we deduce from~\eqref{Eq: difference quotients uniformly bounded I} that there is a sequence $(h_n)_n$ of positive numbers such that $h_n$ converges to $0$, and such that $D^j_{h_n} u$ converges to a weak limit point $v\in \Lt^q(w;\Wx^{1,p})$. We claim that, for almost every $s\in (0,T)$, the function $v(s)$ is the $j$th weak derivative in the $x$-variable of $u(s)$. Indeed, it follows from the \enquote{integration by parts}-identity
	\begin{align}
		\int_{\R^d} (D^j_h f) g \d x = - \int_{\R^d} f (D^j_{-h} g) \d x,
	\end{align}
	which is a simple consequence of translation in the integral, that one has, for $\varphi\in \Cont_0^\infty(\R^d)$ and $\psi \in \Cont_0^\infty(0,T)$, the identity
	\begin{align}
		-\int_0^T \int_{\R^d} \partial_j \varphi(x) u(s,x) \d x \; \psi(s) \d s &= -\lim_n \int_0^T \int_{\R^d} D^j_{-h_n} \varphi(x) u(s,x) \d x \; \psi(s) \d s \\
		&= \lim_n \int_0^T \int_{\R^d} \varphi(x) D^j_{h_n} u(s,x) \d x \; \psi(s) \d s.
	\end{align}
	Integration against $\varphi(x)\psi(s)$ gives rise to a functional on $\Lt^q(w;\Lx^p)$, hence weak convergence of $D^j_{h_n} u$ identifies the latest limit with
	\begin{align}
		\int_0^T \int_{\R^d} \varphi(x) v(s,x) \d x \; \psi(s) \d s.
	\end{align}
	Finally, the fundamental lemma of the calculus of variations shows
	\begin{align}
		- \int_{\R^d} \partial_j \varphi(x) u(s,x) \d x  = \int_{\R^d} \varphi(x) v(s,x) \d x \qquad \text{for almost every } s\in (0,T),
	\end{align}
	which reveals $\partial_j u(s,x) = v(s,x)$ for almost every $s\in (0,T)$ and $j=1,\dots,d$. But as $v\in \Lt^q(w;\Wx^{1,p})$, the lifting property for Sobolev spaces shows $u\in \Lt^q(w;\Wx^{2,p})$.

	\textbf{Step 3}: Uniform bounds using a commutator argument.
	The current step is the essence of this proof, filling in the details of the heuristic given in the roadmap in Section~\ref{Subsec: Roadmap}. As in Step~2, we continue to work with the regularized coefficients, but still omit the subscript $n$ in the notation. However, now we will also rely on the properties established in Step~1 that are uniform in $n$.

	Note that $\partial_x^\alpha$ commutes with $\nabla_x$ and $\partial_t$, where the former fact is a consequence of its definition as a Fourier multiplier.

	Our goal is to show that $\partial_x^\alpha u$ is a weak $(p,q)$-solution to some admissible right-hand side. Note that $\partial_x^\alpha u \in \Lt^q(w;\Wx^{1,p})$, owing to the higher spatial regularity of $u$ established in Step~2, which allows us to plug this term into the equation. That being said, calculate
	\begin{gather}
		\int_0^T - \varphi'(s) ( \partial_x^\alpha u(s) \SP g ) + \varphi(s) b_s(\partial_x^\alpha u(s), g) + \varphi(s) \shift ( \partial_x^\alpha u(s) \SP g )\d s \\
		= \int_0^T - \varphi'(s) ( u(s) \SP \partial_x^\alpha g ) + \varphi(s) b_s(u(s), \partial_x^\alpha g) + \varphi(s) \shift ( u(s) \SP \partial_x^\alpha g ) \d s \\
		+ \int_0^T \varphi(s) \bigl[ b_s(\partial_x^\alpha u(s), g) - b_s(u(s), \partial_x^\alpha g) \bigr] \d s.
	\end{gather}
	Note that $\partial^\alpha_x g \in \Wx^{1,p'}$ since $g$ is smooth and compactly supported. Hence, in the light of Remark~\ref{Rem: weak solution}~\ref{Item: more general g}, use the equation for $u$, and expand the definition of $b_s$, to rewrite the last expression as
	\begin{gather}
		\int\limits_0^T \varphi(s) ( f(s) \SP \partial_x^\alpha g ) \d s + \int\limits_0^T \varphi(s) \int\limits_{\R^d} B(s,x) \nabla \partial_x^\alpha u(s) \cdot \overline{\nabla g} - B(s,x) \nabla u(s) \cdot \overline{\nabla \partial_x^\alpha g} \d x \d s \\
		\eqqcolon \mathrm{I} + \mathrm{II}.
	\end{gather}
	We have to check that the terms $\mathrm{I}$ and $\mathrm{II}$ are induced by right-hand sides in $\Lt^q(w;\Wx^{-1,p})$. For term~$\mathrm{I}$, this is a direct consequence of the mapping properties of $\partial_x^\alpha$ described in Definition~\ref{Def: fractional derivative}, and the $\Lt^q(w;\Wx^{-1,p})$-norm can be controlled by $\| f\|_{\Lt^q(w;\Lx^p)}$.

	Let us proceed with term~$\mathrm{II}$. Keep in mind that $B(s,x)$ is Lipschitz in $x$, and thus is a multiplier on $\Wx^{1,p}$. We use this fact and higher regularity of $u$ from Step~2 to commute $\partial_x^\alpha$ with $\nabla_x$ to rewrite the integral over $\R^d$ in $\mathrm{II}$ as
	\begin{align}
		\int_{\R^d} B(s,x) \nabla \partial_x^\alpha u(s) \cdot \overline{\nabla g} - B(s,x) \nabla u(s) \cdot \overline{\nabla \partial_x^\alpha g} \d x = \int_{\R^d} \Bigl[ B(s,x), \partial_x^\alpha \Bigr] \nabla u(s) \cdot \overline{\nabla g} \d x.
	\end{align}
	Now, we apply the commutator estimates from Lemma~\ref{Lem: commutator estimate} for all times $s$. We only present the case $p < \nicefrac{d}{\alpha}$, the other case is even easier. Keep in mind that $B(s, \cdot)$ is smooth and bounded by the regularization in Step~1. Hence, the latter part of Lemma~\ref{Lem: commutator estimate} along with H\"{o}lder's inequality show that term $\mathrm{II}$ is induced by an $\Lt^q(w;\Wx^{-1,p})$ function as well. This time, the $\Lt^q(w;\Wx^{-1,p})$-norm is controlled by $C_\eps \| u \|_{\Lt^q(w;\Wx^{1,p})} + \eps \| \partial_x^\alpha u \|_{\Lt^q(w;\Wx^{1,p})}$ for any $\eps > 0$, where implicit constants depend on the Sobolev regularity of the coefficients, which is also under control by Step~1. Observe that we have used once more that $\partial_x^\alpha$ and $\nabla$ commute.

	Under the line, Theorem~\ref{Thm: weak solutions} gives $\partial_x^\alpha u \in \Lt^q(w;\Wx^{1,p})$ with estimate
	\begin{align}
	\label{Eq: final bound from Step 3 in higher regularity}
		\| \partial_x^\alpha u \|_{\Lt^q(w;\Wx^{1,p})} \lesssim \| f\|_{\Lt^q(w;\Lx^p)} + C_\eps \| u \|_{\Lt^q(w;\Wx^{1,p})} + \eps \| \partial_x^\alpha u \|_{\Lt^q(w;\Wx^{1,p})}.
	\end{align}
	Choosing $\eps$ sufficiently small, we can absorb the term $\eps \| \partial_x^\alpha u \|_{\Lt^q(w;\Wx^{1,p})}$ into the left-hand side. Finally, apply Theorem~\ref{Thm: weak solutions} once more, but this time for $u$ instead of $\partial_x^\alpha u$, to deduce
	\begin{align}
		\| u \|_{\Lt^q(w;\Wx^{1+\alpha,p})} &\lesssim \| u \|_{\Lt^q(w;\Lx^{p})} + \| \partial_x^\alpha u \|_{\Lt^q(w;\Wx^{1,p})} \\
		&\lesssim \| f\|_{\Lt^q(w;\Lx^p)} + C_\eps \| u \|_{\Lt^q(w;\Wx^{1,p})} \\
		&\lesssim \| f\|_{\Lt^q(w;\Lx^p)}.
	\end{align}

	\textbf{Step 4}: Taking the limit in Step 1.
	The solutions $u_n$ to the regularized equations from Step~1 satisfy the identity
	\begin{align}
	\begin{split}
	\label{Eq: Weak formulation for limit argument}
		&\int\limits_0^T \varphi'(s) ( u_n(s) \SP g ) + \varphi(s) ( f(s) \SP g ) - \varphi(s) \shift ( u_n(s) \SP g ) \d s \\
		={} &\int\limits_0^T \varphi(s) \int\limits_{\R^d} B_n(s,x) \nabla u_n(s) \cdot \overline{\nabla g} \d x \d s.
	\end{split}
	\end{align}
	Moreover, we have seen in Step~3 that $\| u_n \|_{\Lt^q(w;\Wx^{1+\alpha,p})} \lesssim \|f\|_{\Lt^q(w;\Lx^p)}$ holds uniformly in $n$. Since $p$ and $q$ are in the reflexive range, we find a subsequence (which we still denote by $u_n$) for which $u_n$ and $\nabla u_n$ converge weakly in $\Lt^q(w;\Lx^p)$ to some limit $v\in \Lt^q(w;\Wx^{1+\alpha,p})$. Moreover, $\| v \|_{\Lt^q(w;\Wx^{1+\alpha,p})} \lesssim \|f\|_{\Lt^q(w;\Lx^p)}$. The former fact directly enables us to pass to the limit
	\begin{align}
		\int\limits_0^T \varphi'(s) ( v(s) \SP g ) + \varphi(s) ( f(s) \SP g ) + \varphi(s) \shift ( v(s) \SP g ) \d s
	\end{align}
	on the left-hand side of~\eqref{Eq: Weak formulation for limit argument}. For the right-hand side, write
	\begin{align}
		\int\limits_0^T \varphi(s) \int\limits_{\R^d} B_n(s,x) \nabla u_n(s,x) \cdot \overline{\nabla g(x)} \d x \d s = \int\limits_0^T \varphi(s) \int\limits_{\R^d} \nabla u_n(s,x) \cdot \overline{B_n(s,x)^* \nabla g(x)} \d x \d s.
	\end{align}
	Clearly, $B_n(s,x)^*$ is uniformly bounded, and, by $\Contx^\eps$ regularity of $B$, one has $B_n(s,x)^* \to B(s,x)^*$ pointwise. Hence, the dominated convergence theorem gives $\varphi(s) B_n(s,x)^* \nabla g(x) \to \varphi(s) B(s,x)^* \nabla g(x)$ strongly in $\Lt^{q'}(w';\Lx^{p'})$. Hence, the right-hand side of~\eqref{Eq: Weak formulation for limit argument} converges to
	\begin{align}
		\int\limits_0^T \varphi(s) \int\limits_{\R^d} B(s,x) \nabla v(s,x) \cdot \overline{\nabla g(x)} \d x \d s.
	\end{align}
	In summary, taking the limit in~\eqref{Eq: Weak formulation for limit argument} results in
	\begin{gather}
		\int\limits_0^T \varphi'(s) ( v(s) \SP g ) + \varphi(s) ( f(s) \SP g ) - \varphi(s) \shift ( v(s) \SP g ) \d s \\
		= \int\limits_0^T \varphi(s) \int\limits_{\R^d} B(s,x) \nabla v(s,x) \cdot \overline{\nabla g(x)} \d x \d s.
	\end{gather}
	This shows that $u$ and $v$ solve the same equation. Uniqueness of solutions leads to $u = v \in \Lt^q(w;\Wx^{1+\alpha,p})$ as desired. The corresponding estimate was already mentioned above.
\end{proof}

\begin{remark}
\label{Rem: Commutation reason for whole space}
	In Step~3, we have used that the fractional derivative can be written as a Fourier multiplier, and hence commutes with $\nabla$. This is the central reason that ties us to the whole-space in the $x$ variable.
	Moreover, the limiting argument in Step~4 relies on the control of the implied constants from Theorem~\ref{Thm: weak solutions}.
\end{remark}

\section{Proof of Theorem~\ref{Thm: main result}}
\label{Sec: proof main result}

Following the plan outlined in the roadmap in Section~\ref{Subsec: Roadmap} we assemble the results from the previous sections to prove Theorem~\ref{Thm: main result}.

\begin{proof}[Proof of Theorem~\ref{Thm: main result}]
	Let $f\in \Lt^q(w;\Lx^p)$. In virtue of Remark~\ref{Rem: weak solution}~\ref{Item: shift equation}, we consider the shifted problem~\eqref{Eq: shifted non-autonomous equation} with $\cB_t = \cL_t$ instead of~\eqref{Eq: non-autonomous equation}. Let $u$ be its unique $(p,q)$-solution from Theorem~\ref{Thm: weak solutions}. We want to show $\cL_t u(t) \in \Lt^q(w;\Lx^p)$ with estimate against $\|f\|_{\Lt^q(w;\Lx^p)}$. This happens in three steps.

	\textbf{Step~1}: Reduction to right-hand sides in~$\Cont_0^\infty(\Lx^p\cap \Lx^2)$.
	Let $(f_n)_n$ be a sequence in $\Cont_0^\infty(\Lx^p\cap \Lx^2)$ that converges to $f$ in $\Lt^q(w;\Lx^p)$. Let $u_n$ be the weak $(p,q)$-solution of~\eqref{Eq: shifted non-autonomous equation} with $\cB_t = \cL_t$ and right-hand side $f_n$ provided by Theorem~\ref{Thm: weak solutions}. Suppose the maximal regularity estimate
	\begin{align}
	\label{Eq: MR estimate for aproximation}
		\| \cL_t u_n(t) \|_{\Lt^q(w;\Lx^p)} \lesssim \| f_n \|_{\Lt^q(w;\Lx^p)}
	\end{align}
	with implicit constant independent of $n$. In the sequel, we allow tacitly passing to subsequences, even without changing the notation. Since the $u_n$ are weak $(p,q)$-solutions, arguing as in Step~4 of the proof of Proposition~\ref{Prop: higher regularity}, we see that $u_n$ converges weakly to $u$ in $\Lt^q(w;\Lx^p)$ and that $\nabla u_n$ converges weakly to $\nabla u$ in $\Lt^q(w;\Lx^p)$. Let $v\in \Lt^{q'}(w';\Lx^{p'}) \cap \Lt^2(\Lx^2)$. Then, in particular, $(\cL_\tt u_n(\tt) \SP v(\tt)) \to (\cL_\tt u(\tt) \SP v(\tt))$ for almost every $\tt$. Consequently, we find by Fatou's lemma and~\eqref{Eq: MR estimate for aproximation} that
	\begin{align}
		\Bigl| \int_0^T \langle \cL_t u(t), v(t) \rangle \d t \Bigr| &\leq \liminf_n \Bigl| \int_0^T \langle \cL_t u_n(t), v(t) \rangle \d t \Bigr| \\
		&\lesssim \liminf_n \| f_n \|_{\Lt^q(\Lx^p)} \| v \|_{\Lt^{q'}(w;\Lx^{p'})} \\
		&= \| f \|_{\Lt^q(w;\Lx^p)} \| v \|_{\Lt^{q'}(w';\Lx^{p'})}.
	\end{align}
	Hence, duality yields $\| \cL_t u(t) \|_{\Lt^q(w;\Lx^p)} \lesssim \| f \|_{\Lt^q(w;\Lx^p)}$, provided we can show~\eqref{Eq: MR estimate for aproximation}.

	\textbf{Step~2}: Treating the first term in~\eqref{Eq: representation formula}.
	We write $u$ and $f$ instead of $u_n$ and $f_n$ for this part to emphasize that this step does not rely on the regularization of the right-hand side.
	Let $v\in \Lt^{q'}(w';\Lx^{p'}) \cap \Lt^2(\Lx^2)$. We aim to estimate $S_1$ by duality. To this end, write
	\begin{align}
	\label{Eq: S1 duality estimate start}
	\begin{split}
		&\Bigl| \int_0^T \int_0^t ((\cL_t+\shift) \e^{-(t-s) (\cL_t + \shift)} (\cL_t - \cL_s) u(s) \SP v(t)) \d s \d t \Bigr| \\
		={} &\Bigl| \int_0^T \int_0^t (u(s) \SP \bigl( (\cL_t+\shift) \e^{-(t-s) (\cL_t + \shift)} (\cL_t - \cL_s) \bigr)^* v(t)) \d s \d t \Bigr|.
	\end{split}
	\end{align}
	For $t$ and $s$ fixed, the operator $(\cL_t+\shift) \e^{-(t-s) (\cL_t + \shift)} (\cL_t - \cL_s)$ maps $\Wx^{1+\alpha,p} \to \Lx^p$ with norm controlled by $|t-s|^{-1+\eps}$ as combining Lemmas~\ref{Lem: estimate kernel of S1 part 1} and~\ref{Lem: estimate kernel of S1 part 2} shows. Consequently, its adjoint maps $\Lx^{p'} \to \Wx^{-1-\alpha,p'}$ with the same bound. Use this together with the $\Wx^{1+\alpha,p}$--$\Wx^{-1-\alpha, p'}$ duality pairing in~\eqref{Eq: S1 duality estimate start} to bound its right-hand side by
	\begin{align}
		&\int_0^T \int_0^t \| u(s) \|_{\Wx^{1+\alpha,p}} \| \bigl( (\cL_t+\shift) \e^{-(t-s) (\cL_t + \shift)} (\cL_t - \cL_s) \bigr)^* v(t) \|_{\Wx^{-1-\alpha,p'}} \d s \d t \\
		\lesssim{} &\int_0^T \int_0^t \| u(s) \|_{\Wx^{1+\alpha,p}} |t-s|^{-1+\eps} \| v(t) \|_{\Lx^{p'}} \d s \d t.
	\end{align}
	By H\"older's inequality, this can be bounded by
	\begin{align}
		\label{Eq: velo}
		\Bigl\| \int_0^t |t-s|^{-1+\eps} \| u(s) \|_{\Wx^{1+\alpha,p}} \d s  \Bigr\|_{\Lt^q(w)} \| v(t) \|_{\Lt^{q'}(w';\Lx^{p'})}.
	\end{align}
	By Lemma~\ref{Lem: weighted convolution} (the convolution kernel $s\mapsto |s|^{-1+\eps}$ is radial, decreasing, and integrable over $(0,T)$) and Proposition~\ref{Prop: higher regularity}, control~\eqref{Eq: velo} by $\| u \|_{\Lt^q(w;\Wx^{1+\alpha,p})} \| v \|_{\Lt^{q'}(w';\Lx^{p'})} \lesssim \| f \|_{\Lt^q(w;\Lx^p)} \| v \|_{\Lt^{q'}(w';\Lx^{p'})}$.
	Hence, duality shows in summary
	\begin{align}
	\label{Eq: S1 estimate}
		\| \int_0^t (\cL_t+\shift) \e^{-(t-s) (\cL_t + \shift)} (\cL_t - \cL_s) u(s) \d s \|_{\Lt^q(w;\Lx^p)} \lesssim \| f \|_{\Lt^q(w;\Lx^p)}.
	\end{align}
	In particular, the above calculation (applied with $v$ constant) shows that $$\int_0^t \| (\cL_t+\shift) \e^{-(t-s) (\cL_t + \shift)} (\cL_t - \cL_s) u(s) \|_{\Lx^p} \d s < \infty \qquad \text{for almost every }t \in (0,T).$$
	Whence, Hille's theorem shows
	\begin{align}
		(\cL_t+\shift) \int_0^t \e^{-(t-s) (\cL_t + \shift)} (\cL_t - \cL_s) u(s) \d s = S_1(u)(t),
	\end{align}
	so that~\eqref{Eq: S1 estimate} translates to
	\begin{align}
		\| S_1(u) \|_{\Lt^q(w;\Lx^p)} \lesssim \| f \|_{\Lt^q(w;\Lx^p)}.
	\end{align}

	\textbf{Step~3}: Treating the second term in~\eqref{Eq: representation formula}.
	Thanks to the reduction to more regular right-hand sides in Step~1, Proposition~\ref{Prop: S2 bdd} directly yields $\| S_2(f_n) \|_{\Lt^q(w;\Lx^p)} \lesssim \| f_n \|_{\Lt^q(w;\Lx^p)} \lesssim \| f \|_{\Lt^q(w;\Lx^p)}$.

	In summary, Steps~2 and~3 in conjunction with Theorem~\ref{Thm: weak solutions} give
	\begin{align}
		\| \cL_t u_n(t) \|_{\Lt^q(w;\Lx^p)} \lesssim \| S_1(u_n) \|_{\Lt^q(w;\Lx^p)} + \| S_2(f_n) \|_{\Lt^q(w;\Lx^p)} + \shift \| u_n \|_{\Lt^q(w;\Lx^p)} \lesssim \| f \|_{\Lt^q(w;\Lx^p)},
	\end{align}
	which is~\eqref{Eq: MR estimate for aproximation}. Hence, $\| \cL_t u(t) \|_{\Lt^q(w;\Lx^p)} \lesssim \| f \|_{\Lt^q(w;\Lx^p)}$ as was demonstrated in Step~1. This completes the proof.
\end{proof}


\nocite{*}
\begin{thebibliography}{11}

\bibitem{AO19}
\textsc{M.~Achache} and \textsc{E.M.~Ouhabaz}.
\newblock {\em Lions' maximal regularity problem with $H^\frac{1}{2}$-regularity in time\/}.
\newblock J. Differential Equations~\textbf{266} (2019), no.~6, 3654--3678.
\newblock \doi{10.1016/j.jde.2018.09.015}

\bibitem{AT87}
\textsc{P.~Acquistapace} and \textsc{B.~Terreni}.
\newblock {\em A unified approach to abstract linear nonautonomous parabolic equations\/}.
\newblock Rend. Sem. Mat. Univ. Padova~\textbf{78} (1987), 47--107.

\bibitem{Adams}
\textsc{R.A.~Adams}, \textsc{J.J.F.~Fournier}.
\newblock {\em Sobolev Spaces\/}.
\newblock Elsevier/Academic Press, Amsterdam, 2003.


\bibitem{Lions-Survey}
\textsc{W.~ Arendt}, \textsc{D.~Dier}, and \textsc{S.~Fackler}.
\newblock {\em J.~L.~{L}ions' problem on maximal regularity\/}.
\newblock Arch. Math. \textbf{109} (2017), no.~1, 59--72.
\newblock \doi{10.1007/s00013-017-1031-6}

\bibitem{Memoirs}
\textsc{P.~Auscher}.
\newblock {\em On necessary and sufficient conditions for {$L^p$}-estimates of
  {R}iesz transforms associated to elliptic operators on {$\mathbb{R}^n$} and
  related estimates\/}.
\newblock Mem. Amer. Math. Soc. \textbf{186} (2007), no.~871.


\bibitem{AE16}
\textsc{P.~Auscher} and \textsc{M.~Egert}.
\newblock {\em On non-autonomous maximal regularity for elliptic operators in divergence form\/}.
\newblock Arch. Math. (Basel) \textbf{107} (2016), no.~3, 271--284.
\newblock \doi{10.1007/s00013-016-0934-y}

\bibitem{Kato}
\textsc{P.~Auscher}, \textsc{S.~Hofmann}, \textsc{M.~Lacey}, \textsc{A.~McIntosh}, and \textsc{P.~Tchamitchian}.
\newblock {\em The solution of the {K}ato square root problem for second order
  elliptic operators on {$\mathbb{R}^n$}\/}.
\newblock Ann. of Math. (2) \textbf{156} (2002), no.~2, 633--654.
\newblock \doi{10.2307/3597201}

\bibitem{Auscher-Tchamitchian}
\textsc{P.~Auscher} and \textsc{P.~Tchamitchian}.
\newblock Square root problem for divergence operators and related topics. Ast\'{e}risque, vol.~249.
\newblock Soci\'{e}t\'{e} math\'{e}matique de France, 1988.


\bibitem{Lp-Kato}
\textsc{S.~Bechtel}.
\newblock {\em $L^p$-estimates for the square root of elliptic systems with mixed boundary conditions II\/}.
\newblock J. Differential Equations~\textbf{379} (2024), 104--124.
\newblock \doi{10.1016/j.jde.2023.09.036}

\bibitem{BMV}
\textsc{S.~Bechtel}, \textsc{C.~Mooney}, and \textsc{M.~Veraar}.
\newblock {\em Counterexamples to maximal regularity for operators in divergence form\/}.
\newblock To appear in Arch. Math. ArXiv preprint available at \url{https://arxiv.org/abs/2401.05550}.

\bibitem{D-to-N}
\textsc{S.~Bechtel} and \textsc{E.M.~Ouhabaz}.
\newblock arXiv preprint, available at \url{https://arxiv.org/abs/2207.09115v1}.

\bibitem{Blasio06}
\textsc{G.D.~Blasio}.
\newblock {\em Maximal $L^p$ regularity for nonautonomous parabolic equations in extrapolation spaces\/}.
\newblock J. Evol. Equ.~\textbf{6} (2006), no.~2, 229--245.
\newblock \doi{10.1007/s00028-006-0241-3}



\bibitem{DZ17}
\textsc{D.~Dier} and \textsc{R.~Zacher}.
\newblock {\em Non-autonomous maximal regularity in Hilbert spaces\/}.
\newblock J. Evol. Equ.~\textbf{17} (2017), no.~3, 883--907.
\newblock \doi{10.1007/s00028-016-0343-5}

\bibitem{Dong-Kim-ARMA}
\textsc{H.~Dong} and \textsc{D.~Kim}.
\newblock{\em On the $L_p$-solvability of higher order parabolic and elliptic systems with {BMO} coefficients\/}.
\newblock Arch. Ration. Mech. Anal \textbf{109} (2011), no.~3, 889--941.
\newblock \doi{10.1007/s00205-010-0345-3}

\bibitem{Dong-Kim-JFA}
\textsc{H.~Dong} and \textsc{D.~Kim}.
\newblock{\em Higher order elliptic and parabolic systems with variably partially BMO coefficients in regular and irregular domains\/}.
\newblock J. Funct. Anal. \textbf{261} (2011), no.~11, 3279--3327.
\newblock \doi{10.1016/j.jfa.2011.08.001}

\bibitem{Dong-Kim-CoV}
\textsc{H.~Dong} and \textsc{D.~Kim}.
\newblock{\em $L_p$ solvability of divergence type parabolic and elliptic systems with partially {BMO} coefficients\/}.
\newblock Calc. Var. Partial Differential Equations \textbf{40} (2011), no.~3--4, 357--389.
\newblock \doi{10.1007/s00526-010-0344-0}

\bibitem{Dong-Kim-TAMS}
\textsc{H.~Dong} and \textsc{D.~Kim}.
\newblock {On $L_p$-estimates for elliptic and parabolic equations with $A_p$ weights\/}.
\newblock Trans. Amer. Math. Soc. \textbf{370} (2018), no.~7, 5081--5130.
\newblock \doi{10.1090/tran/7161}


\bibitem{Evans}
\textsc{L.C.~Evans}.
\newblock {P}artial {D}ifferential {E}quations. Second edition. Graduate Studies in Mathematics, vol.~19.
\newblock Amer.\@ Math.\@ Soc.\@, Providence, RI, 2010.
\newblock \doi{10.1090/gsm/019}

\bibitem{Fackler15}
\textsc{S.~Fackler}.
\newblock arXiv preprint, available at \url{https://arxiv.org/abs/1511.06207v3}.

\bibitem{Fackler16}
\textsc{S.~Fackler}.
\newblock {\em  J.-L.~Lions' problem concerning maximal regularity of equations governed by non-autonomous forms\/}.
\newblock Ann. Inst. H. Poincar\'{e} C Anal. Non Lin\'{e}aire~\textbf{34} (2017), no.~3, 699--709.
\newblock \doi{10.1016/j.anihpc.2016.05.001}

\bibitem{Fackler18}
\textsc{S.~Fackler}.
\newblock {\em Nonautonomous maximal $L^p$-regularity under fractional Sobolev regularity in time\/}.
\newblock Anal. PDE~\textbf{11} (2018), no.~5, 1143--1169.
\newblock \doi{10.2140/apde.2018.11.1143}

\bibitem{Nick}
\textsc{S.~Fackler}, \textsc{T.P.~Hyt\"onen}, and \textsc{N.~Lindemulder}.
\newblock {\em Weighted estimates for operator-valued Fourier multipliers\/}.
\newblock Collect. Math.~\textbf{71} (2020), no.~3, 511--548.
\newblock \doi{10.1007/s13348-019-00275-0}

\bibitem{GV17a}
\textsc{C.~Gallarati} and \textsc{M.C.~Veraar}.
\newblock {\em Maximal regularity for non-autonomous equations with measurable dependence on time\/}.
\newblock Potential Anal.~\textbf{46} (2017), no.~3, 527--567.
\newblock \doi{10.1007/s11118-016-9593-7}


\bibitem{RubioDeFrancia}
\textsc{J.~Garc\'{i}a-Cuerva} and \textsc{J.~Rubio de Francia}.
\newblock Weighted norm inequalities and related topics. North-Holland Mathematics Studies, vol.~116.
\newblock North-Holland Publishing Co., Amsterdam, 1985.

\bibitem{HO15}
\textsc{B.~Haak} and \textsc{E.M.~Ouhabaz}.
\newblock {\em Maximal regularity for non-autonomous evolution equations\/}.
\newblock Math. Ann.~\textbf{363} (2015), no.~3--4, 1117--1145.
\newblock \doi{10.1007/s00208-015-1199-7}

\bibitem{Haase}
\textsc{M.~Haase}.
\newblock The {F}unctional {C}alculus for {S}ectorial {O}perators. Operator Theory: Advances and Applications, vol.~169.
\newblock Birkh{\"a}user, Basel, 2006.
\newblock \doi{10.1007/3-7643-7698-8}


\bibitem{HM00}
\textsc{M.~Hieber} and \textsc{S.~Monniaux}.
\newblock {\em Pseudo-differential operators and maximal regularity results for non-autonomous parabolic equations\/}.
\newblock Proc. Amer. Math. Soc.~\textbf{128} (2000), 1047--1053.
\newblock \doi{10.1090/S0002-9939-99-05145-X }

\bibitem{AiBSVolI}
\textsc{T.~Hyt\"onen}, \textsc{J.~van Neerven}, \textsc{M.~Veraar}, and \textsc{L.~Weis}.
\newblock{\em Analysis in Banach spaces. Vol. I. Martingales and Littlewood-Paley theory\/}.
\newblock Springer, Cham, 2016.
\newblock \doi{10.1007/978-3-319-48520-1}

\bibitem{HP}
\textsc{T.~Hyt\"onen} and \textsc{P.~Portal}.
\newblock {\em Vector-valued multiparameter singular integrals and pseudodifferential operators\/}.
\newblock Adv. Math.~\textbf{217} (2008), no.~2, 519--536.
\newblock \doi{10.1016/j.aim.2007.08.002}

\bibitem{KMM}
\textsc{N.~Kalton}, \textsc{S.~Mayboroda}, and \textsc{M.~Mitrea}.
\newblock Interpolation of Hardy-Sobolev-Besov-Triebel-Lizorkin spaces and applications to problems in partial differential equations. In \emph{Interpolation theory and applications}, volume 445 of Contemp. Math., 121--177.
\newblock Amer. Math. Soc., Providence, RI, 2007.
\newblock \doi{10.1090/conm/445/08598}

\bibitem{KW}
\textsc{P.C.~Kunstmann and L.~Weis}.
Maximal $L_p$-regularity for parabolic equations, Fourier multiplier theorems and $H^\infty$-functional calculus. In \emph{Functional analytic methods for evolution equations}, Lecture Notes in Math. 1855, 65--311.
\newblock Springer, Berlin, 2004.
\newblock \doi{10.1007/978-3-540-44653-8_2}



\bibitem{MV}
\textsc{M.~Meyries} and \textsc{M.C.~Veraar}.
\newblock {\em Pointwise multiplication on vector-valued function spaces with power weights\/}.
\newblock J. Fourier Anal. Appl.~\textbf{21} (2015), no.~1, 95--136.
\newblock \doi{10.1007/s00041-014-9362-1}

\bibitem{OS10}
\textsc{E.M.~Ouhabaz} and \textsc{C.~Spina}.
\newblock {\em Maximal regularity for non-autonomous {S}chr\"{o}dinger type equations\/}.
\newblock J. Differential Equations~\textbf{248} (2010), no.~7, 1668--1683.
\newblock \doi{10.1016/j.jde.2009.10.004}

\bibitem{PS06}
\textsc{P.~Portal} and \textsc{\v{Z}.~\v{S}trkalj}.
\newblock {\em Pseudodifferential operators on Bochner spaces and an application\/}.
\newblock  Math. Z.~\textbf{253} (2006), no.~4, 805--819.
\newblock \doi{10.1090/S0002-9939-99-05145-X }


\bibitem{PS01}
\textsc{J.~Prüss} and \textsc{R.~Schnaubelt}.
\newblock {\em Solvability and maximal regularity of parabolic evolution equations with coefficients continuous in time\/}.
\newblock J.~Math.~Anal.~Appl.\@ \textbf{256} (2001), no.~2, 405--430.
\newblock \doi{10.1006rjmaa.2000.7247}

\bibitem{Critical-Spaces}
\textsc{J.~Prüss}, \textsc{G.~Simonett}, and \textsc{M.~Wilke}.
\newblock {\em Critical spaces for quasilinear parabolic evolution equations and applications\/}.
\newblock  J. Differential Equations\@ \textbf{264} (2018), no.~3, 2028--2074.
\newblock \doi{10.1016/j.jde.2017.10.010}

\bibitem{Samko}
\textsc{S.G.~Samko}, \textsc{A.A.~Kilbas}, and \textsc{O.I.~Marichev}.
\newblock Fractional Integrals and Derivatives. Theory and Applications.
\newblock Gordon and Breach Science Publishers, Yverdon, 1993.



\bibitem{Stein}
\textsc{E.M.~Stein}.
\newblock Singular integrals and differentiability properties of functions.
\newblock Princeton University Press, Princeton, 1970.
\newblock \doi{10.1515/9781400883882}

\bibitem{Strichartz}
\textsc{R.S.~ Strichartz}.
\newblock {\em Multipliers on fractional Sobolev spaces\/}.
\newblock  J. Math. Mech.\@ \textbf{16} (1967), 1031--1060.

\bibitem{Triebel}
\textsc{H.~Triebel}.
\newblock Interpolation {T}heory, {F}unction {S}paces, {D}ifferential
  {O}perators.  North-Holland Mathematical Library, vol.~18.
\newblock North-Holland Publishing, Amsterdam, 1978.


\end{thebibliography}
\end{document}